\documentclass[11pt,twoside]{article}
\usepackage{amsthm}
\usepackage{latexsym}
\usepackage{amssymb,amsbsy,amsmath,amsfonts,amssymb,amscd}
\usepackage{color}
\usepackage{subcaption}
\usepackage{float,url}
\usepackage{cancel}
\usepackage{graphicx}
\usepackage[colorlinks,linkcolor=red,citecolor=blue]{hyperref}
\usepackage{enumitem}

\numberwithin{equation}{section}

\newcommand{\commentout}[1]{}

\newcommand{\R}{\mathbb{R}}
\newcommand{\N}{\mathbb{N}}

\newcommand {\eps}  {\varepsilon}

\newcommand {\sgn} { {\rm sgn} }

\newcommand {\f}   {\frac}
\newcommand {\p}   {\partial}
\newcommand{\beq}{\begin{equation}}
\newcommand{\eeq}{\end{equation}}
\newtheorem{theorem}{Theorem}[section]
\newtheorem{lemma}[theorem]{Lemma}

\newtheorem{remark}[theorem]{Remark}
\newtheorem{proposition}[theorem]{Proposition}
\newtheorem{corollary}[theorem]{Corollary}

\title{A stiff limit of non-homogeneous conservation laws\\ for crowd motion modeling}

\author{Nicolas Masson \thanks{Université Paris-Saclay, Laboratoire de Mathématiques d'Orsay, 91400 Orsay}\and
Beno\^ \i t Perthame  \thanks{Sorbonne Université, Universit\'{e} de Paris Cit{\'e}, CNRS, Inria, Laboratoire Jacques-Louis Lions, 75005 Paris}
\and Filippo Santambrogio \thanks{Universite Claude Bernard Lyon 1, ICJ UMR5208, CNRS, Ecole Centrale de Lyon, INSA Lyon, Universit\'e Jean Monnet, 69622
Villeurbanne, France}
}
\date{\today}

\begin{document}
\maketitle
\pagestyle{plain}
\pagenumbering{arabic}

\begin{abstract} 
We propose a new approach for crowd motion models where the density constraint can only slow down the motion of each agent, with no effect on those agents who are not in a saturated area or who have no saturated density ``in front'' of them. This is done by means of a limit of conservation laws inspired by the equations used for traffic as in \textit{Follow the leader}-type models. We study the asymptotics of the solutions of these conservation laws in a certain asymptotic regime, and obtain a PDE at the limit of a whole new type. One of the main goals of the paper is to prove uniform BV estimates on the density, and thus strong compactness to prove the existence of solutions to this limit equation. We also discuss the qualitative behavior of solutions, provide numerical illustrations both in dimension $1$ and $2$, and establish the new entropy inequalities associated with this limit equation.
\end{abstract} 
\vskip .7cm

\noindent{\makebox[1in]\hrulefill}\newline
\section{Introduction}

The mathematical modeling of crowd motion is a well-established field encompassing a wide range of approaches. A comprehensive overview of the existing models can be found in \cite{helbing_pedestrian_2010}. A common classification distinguishes models according to the scale of description—\textit{microscopic} or \textit{macroscopic}—and the manner in which congestion constraints are incorporated—\textit{hard} or \textit{soft}. In particle-based descriptions, where individual agents are explicitly represented, the approach is referred to as \textit{microscopic}. In contrast, when the crowd is described in terms of a continuous density field, the model is termed \textit{macroscopic}. Concerning congestion, a constraint is said to be \textit{hard} when it induces a qualitative change in the dynamics upon being reached, typically enforcing a strict upper bound on the density. It is referred to as \textit{soft} when its influence is introduced progressively, with the governing equations smoothly penalizing high-density regions as the system approaches the congestion threshold.

Among macroscopic models with a hard congestion constraint, our starting point is the one proposed by Maury, Roudneff-Chupin and the third author in \cite{roudneff_modelisation_nodate,maury_macroscopic_2010}. In this model, the crowd is represented by a density $\rho$ that remains below a threshold normalized to $1$, and individuals aim to move according to a desired velocity field $U$. Since $U$ may favor concentration (i.e. $\nabla\cdot U<0$), the density $\rho$ is advected by an effective velocity field $u$, defined as the $L^2$-projection of $U$ onto the cone $C_\rho$ of admissible velocity fields, namely those whose divergence is nonnegative on the saturated region $\{\rho=1\}$. The macroscopic model with hard constraint reads

\begin{align*}
\left\{\begin{array}{l}
        \partial_t \rho + \nabla \cdot (\rho u) = 0,  \\
        u = P_{C_\rho} U,
    \end{array}
    \right.
\end{align*}

In the case when $U = - \nabla D$, existence of solutions was established in \cite{maury_macroscopic_2010} using an approximation scheme known as the JKO scheme, introduced by Jordan, Kinderlehrer, and Otto in \cite{jordan_variational_1998}. In this setting, the analysis exploits the underlying gradient-flow structure of the PDE.

A shortcoming of this formulation is that it allows some individuals to move faster than their desired velocity, which contradicts certain empirical observations on crowd motion. In order to overcome this issue, and following an idea introduced in \cite{reda_crowd_nodate}, we study a variant of the previous model by introducing a \textit{hierarchy} between individuals: where the density is strictly less than one, individuals move freely; where $\rho = 1$, individuals closest to the exit move freely, while those behind adapt their velocity. In this way, we obtain a macroscopic model with hard congestion constraint of a new kind, which formally solves a PDE system that can be formulated as a free boundary problem 

\begin{figure}[t]
\center
\includegraphics[scale = 0.6]{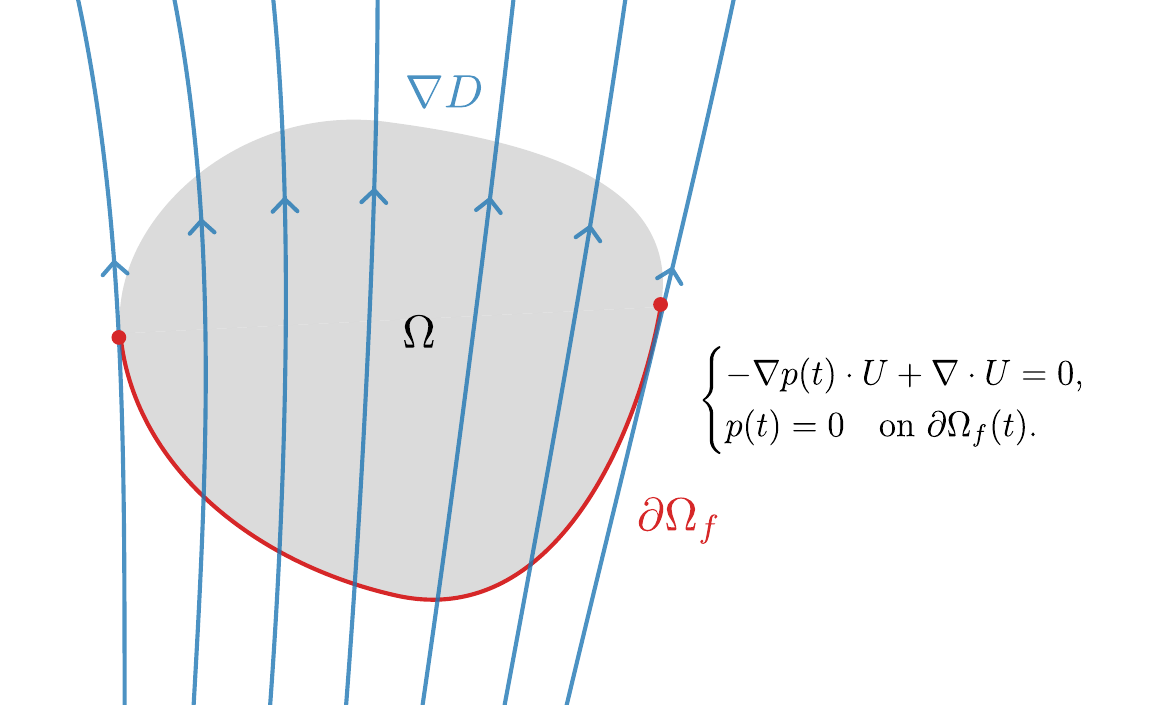}
\caption{Pressure at time $t$ satisfies a stationary transport equation.}
\label{fig:transport_pressure}
\end{figure}

\begin{equation}
\label{eq:FBF}
\left\{\begin{aligned}
 \partial_t \rho + \nabla \cdot (\rho (1-p)U) &= 0, \\
 p(1-\rho) &= 0, \quad 0 \le p \le 1,\\
  \nabla \cdot ((1-p)U) &= 0 \text{ on } \Omega(t), \\
  p &= 0 \text{ on } \partial \Omega_f(t),
\end{aligned}\right.
\end{equation} 
where $p$ plays the role of a pressure, 
\begin{equation} \label{def:Omega}
\Omega(t) = \mathring{\Sigma}(t)\text{ with } \Sigma(t) = \{x \in \mathbb{R}^d \,:\, \rho(t,x) = 1\}
\end{equation}
is the interior of the saturated region,  and $\partial \Omega_f(t)$ denotes the \textit{frontal part} of the boundary $\partial \Omega(t)$ - that is, the points of $\partial \Omega(t)$ through which individuals would exit $\Omega(t)$ if they were moving according to $U$ (see Fig. \ref{fig:transport_pressure}).

To better visualize the expected behavior, we consider a simple one-dimensional example. Take the initial condition
\[
\rho^0 = \mathbf{1}_{[0,x^+(0)]},
\]
with a desired velocity field that is positive and decreasing. Then we expect the solution of be (see Figure~\ref{fig:FBF1d})

\begin{figure}[t]
\centering
\includegraphics[width = \linewidth]{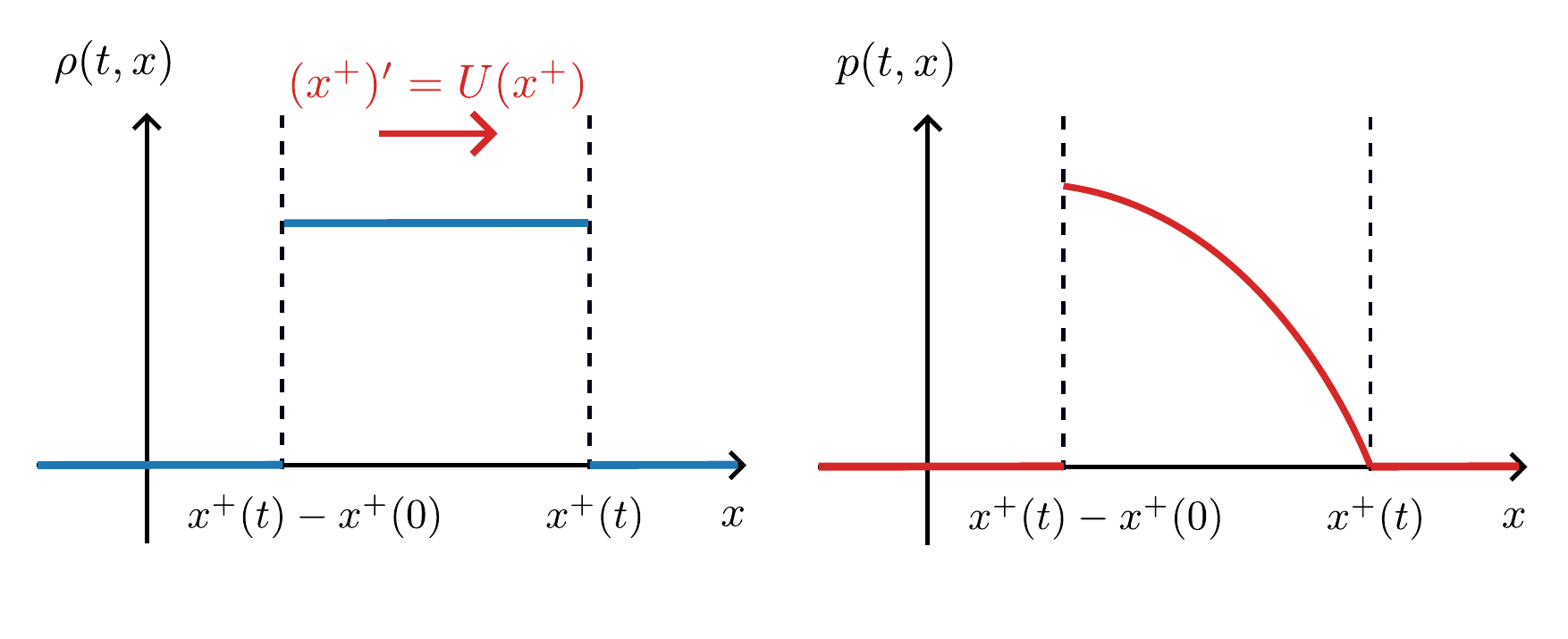}
\caption{Expected solution in a simple 1D case}
\label{fig:FBF1d}
\end{figure}

\[
\rho(t,x) = \begin{cases}
    1 \quad \text{for } x^+(t) -x^+(0) <x<x^+(t) \\
    0 \quad \text{elsewhere},
\end{cases}
\]
and 
\[
p(t,x) = \begin{cases}
1-\f{U(x^+(t))}{U(x)} \quad \text{for } x^+(t) -x^+(0) <x<x^+(t)\\
0 \quad \text{elsewhere},
\end{cases}
\]
where

\[
(x^+)'(t) = U(x^+(t)).
\]
The pressure $p$ simply corrects $U$ so that the effective velocity $(1-p) U$ is constant inside the saturated region  and equal to the desired velocity of the individual located at the front. In other words, the behavior is the following: the first person in the queue moves at its desired speed, while those behind would like to move faster, but congestion prevents them from doing so. They therefore adapt and move at the same speed as the person at the front.

We now consider a more interesting example, where the saturated region has initially two connected components (see Fig.~\ref{fig:FBF1D_AHP}). The initial density is given by

\begin{figure}
    \centering
    \includegraphics[width=\linewidth]{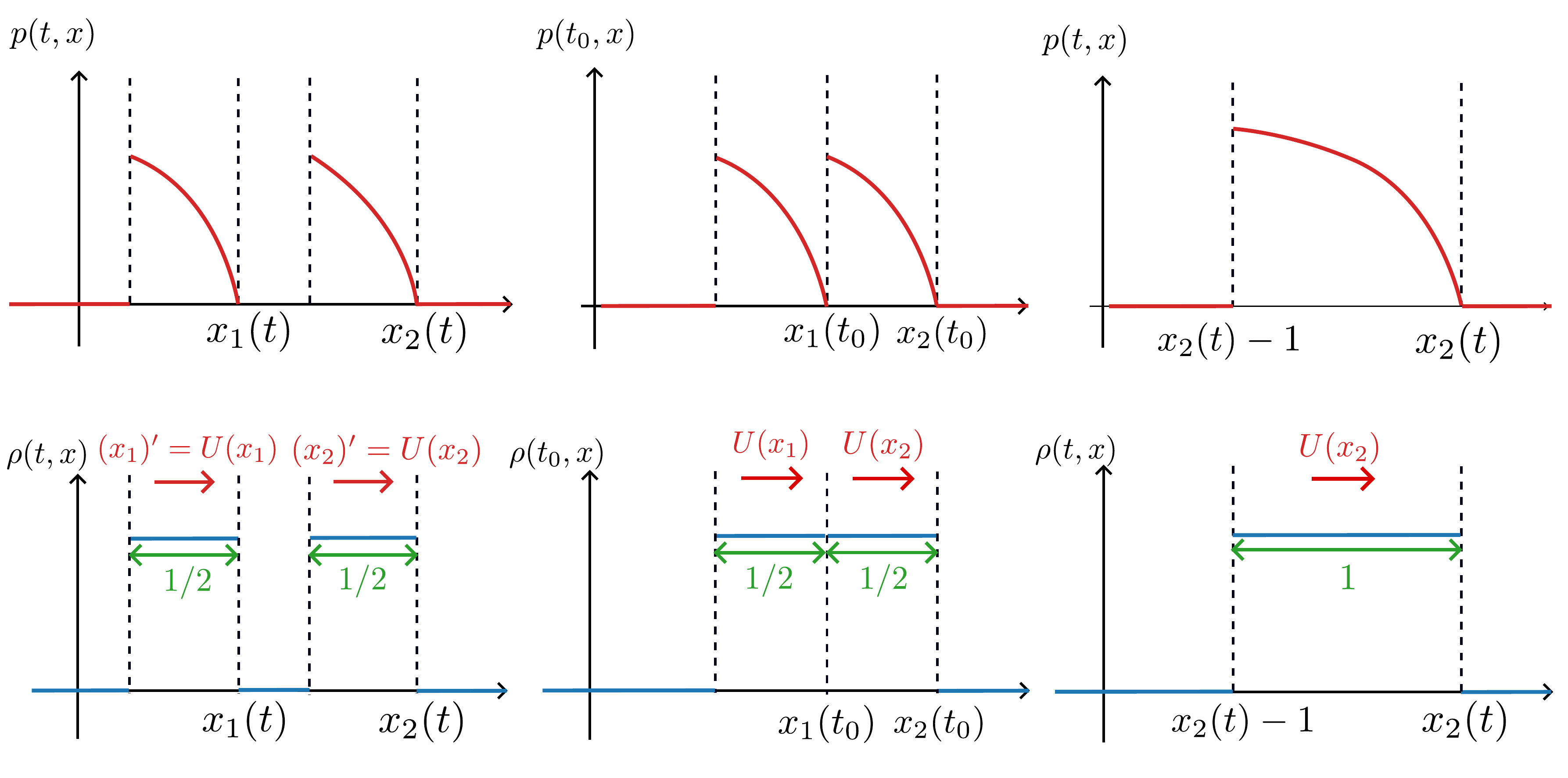}
    \caption{Density (second row) and pressure (first row) when two saturated areas collide at $t_0$.}
    \label{fig:FBF1D_AHP}
\end{figure}

\[
\rho^0 = \mathbf{1}_{[-1,-1/2]}+\mathbf{1}_{[1/2, 1]}.
\]
Then, the expected solution is

\[
\rho(t,x) = \mathbf{1}_{[x_1(t)-1/2,x_1(t)]} +\mathbf{1}_{[x_2(t)-1/2, x_2(t)]},
\]
where

\[
\begin{cases}
    x_1'(t)=U(x_1(t)),
    \\
    x_1(0) = -1/2,
\end{cases}
\qquad \text{and}\qquad \qquad 
\begin{cases}
    x_2'(t)=U(x_2(t)),\\
    x_2(0) = 1/2.
\end{cases}
\]
Then, provided that $U$ is Lipschitz continuous, we know from the Cauchy-Lipschitz theorem that $x_1(t)<x_2(t)$ for all~$t$. However, if $U$ is contracting enough (for instance, if $U' \le \alpha <0$), then $x_1(t)$ reaches $x_2(t)-1/2$ in finite time $t_0$ (see Figure~\ref{fig:FBF1D_AHP}). In other words, the person in the front of the left connected component of the saturated area reaches the person in the back of the right connected component. At that time, we expect a singularity in the behavior : the person at $x_1$, after having priority for a certain period of time, suddenly loses priority over the person at $x_2$. It means that his velocity suddenly changes from $U(x_1)$ to $U(x_2)$. The singularity is even more dramatic for the pressure : before the contact-time we have

\[
p(t,x) = \left(1-\f{U(x_1(t))}{U(x)}\right)\mathbf{1}_{[x_1-1/2, x_1]} + \left(1-\f{U(x_2(t))}{U(x)}\right)\mathbf{1}_{[x_2-1/2 , x_2]},
\]
but as soon as $x_1$ reaches $x_2 -1/2$, it becomes

\[
p(t,x) = \left(1-\f{U(x_2(t))}{U(x)}\right)\mathbf{1}_{[x_2(t)-1,x_1(t)]}.
\]

Yet, keeping these difficulties in mind, \eqref{eq:FBF} corresponds - at least formally - in several dimensions to a model where:

\begin{itemize}[label=$\bullet$]

\item The density is advected by the effective velocity field $u = (1-p)U$, obtained by multiplying the desired velocity $U$ by a scalar correction factor.
\item The effective speed is less than or equal to the desired speed, so the main drawback of the model in \cite{maury_macroscopic_2010} is resolved.
\item The correction factor $p$, which we call a pressure (although it should not be interpreted as a physical pressure), is active only in the saturated region.
\item The pressure satisfies a transport equation in $\Omega(t)$, together with an appropriate boundary condition on $\partial \Omega_f(t)$.
\item Overall, the model satisfies the guiding principle: where $\rho < 1$, the density is advected by the desired velocity $U$; where $\rho = 1$, it is advected by the effective velocity $(1-p)U$. Moreover, individuals located on the frontal boundary $\partial \Omega_f(t)$ still move at their desired speed (since $p=0$ there), while those behind adapt according to a transport equation for the pressure, whose characteristics are the streamlines of $-U$.

\end{itemize}

This system (\ref{eq:FBF}) is composed of three different equations, which we denote as follows for future reference : first the continuity equation 

\begin{equation}
\label{eq:continuity}
\partial_t \rho + \nabla \cdot (\rho(1-p)U) = 0, \qquad \text{(continuity equation),}
\end{equation}
then the law of state between $\rho$ and $p$

\begin{equation}
\label{eq:state}
 p(1-\rho) = 0, \quad 0 \le p \le 1,  \qquad \qquad \text{(law of state)},
\end{equation}
and finally the complementarity equation satisfied by $p$ on $\Omega(t)$ - which can be seen  as a transport equation 

\begin{equation}
\label{eq:complementarity}
\nabla \cdot ((1-p)U) = 0 \text{ on } \Omega(t), \qquad \text{(complementarity equation)},
\end{equation}
with its boundary condition 

\begin{equation}
\label{eq:boundary}
p = 0 \text{ on } \partial \Omega_f(t).
\end{equation}

This PDE system appears to reflect our modeling objectives. However, as stated, there is no guarantee that it is wellposed. Indeed, the objects we introduce — such as the pressure \(p\), the saturated zone \(\Omega\), and the frontal boundary \(\partial \Omega_f\) — may have very low regularity, making them ill-suited for giving a meaning to the PDE system, even in a weak or distributional sense.

Following an idea suggested by Bertrand Maury, we prove the existence of solutions to (\ref{eq:FBF}) by approximating the equation with a sequence of scalar conservation laws and showing that the corresponding solutions converge to a solution of (\ref{eq:FBF}). For readers who may be unfamiliar with scalar conservation laws, we remind that they are PDEs of the form

\begin{align}
\label{eq:conservation_generic}
\partial_t \rho + \nabla \cdot (F(\rho,x)) = 0,
\end{align}
and have been extensively studied — see \cite{dafermos_constantine_m_hyperbolic_2016} or \cite{serre_systemes_1996} for comprehensive presentations of the theory.

For this type of first order evolutionary PDEs, characteristics generically intersect, thus generating discontinuities (shocks), which makes it necessary to consider weak formulations of the equation.
However, weak formulations alone do not ensure uniqueness. Uniqueness is recovered by requiring that $\rho$ satisfies a family of differential inequalities known as entropy inequalities. In his seminal paper \cite{kruzkov_first_1970}, Kruzhkov proved existence and uniqueness of entropy solutions — that is, weak solutions satisfying the entropy inequalities — for conservation laws under mild regularity assumptions on the flux function~$F$. An alternative approach, which additionally yields Sobolev regularity of the solutions in all dimensions, is the kinetic formulation \cite{MR2064166}. The particular case at hand (see \eqref{eq:CL} below) is studied in \cite{PeDa2009} for a less regular potential $U=\nabla V$.

One of the best-known scalar conservation law in applied mathematics arises from the \textit{Lighthill-Whitham-Richards} (LWR) traffic model introduced in \cite{lighthill_kinematic_1997} and \cite{richards_shock_1956}, which reads

\begin{equation}
\left\{\begin{aligned}
 &\partial_t \rho + \partial_x (F(\rho)) = 0, \\
 &F(\rho) = \rho(1-\rho)v_{max}.
\end{aligned}\right.
\end{equation}
This model expresses that the vehicle density \(\rho\) is advected with velocity \(v = (1-\rho)v_{max}\), equal to the speed limit \(v_{max}\) modulated by a correction factor that reduces the speed as the density approaches the threshold \(\rho = 1\). It can be shown (see, for instance, \cite{rossi_justification_2014} or \cite{francesco_rigorous_2015}) that this equation is the macroscopic counterpart of \textit{Follow-the-Leader} (FTL) models, in which each vehicle aims to reach the speed limit but adapts its velocity according to the distance to the vehicle directly ahead. Mathematically, FTL models prescribe that the position \(x_i\) of the \(i\)-th vehicle satisfies an ODE of the form
\[
x_i' = \varphi_\delta(x_{i+1}-x_i),
\]
where \(\delta\) is a stiffness parameter accounting for driver's responsiveness (see Fig. \ref{fig:FTL}).

This framework provides a natural starting point for our model and fits well within the classical conservation law theory. However:

\begin{itemize}[label=$\bullet$]

\item In the LWR model, congestion is handled in a \textit{soft} manner: the velocity is smoothly penalized as the density approaches \(1\). While this is reasonable for vehicular traffic — where drivers anticipate congestion and adjust their speed accordingly — it may be less appropriate for the type of saturation effects we aim to model.

\item The LWR model is one-dimensional, and much of the corresponding conservation law theory is developed in that setting, whereas a two-dimensional model would be more appropriate for pedestrian dynamics.

\item Many classical results assume that the flux \(F\) does not depend explicitly on space. In contrast, we would like to prescribe a desired velocity field that depends on the individual's position relative to the exit.

\end{itemize}

Let us explain the approach, and the difficulties lying ahead. To approximate the system (\ref{eq:FBF}) corresponding to our target model, we consider the entropic solution of the LWR-type equation

\begin{equation}
\label{eq:CL}
\left\{\begin{aligned}
 &\partial_t \rho + \nabla \cdot (F_k(\rho)U) = 0, \\
 &F_k(\rho) = \rho(1-\rho^k),\\
 &0 \le \rho(t=0) \le 1, \quad \int\rho(t=0) = 1
\end{aligned}\right.
\end{equation}
and pass to the limit \(k \to +\infty\). In other words, we study the macroscopic counterpart of a model in which each agent adapts its speed according to the distance to the one in front, with a shorter characteristic interaction distance as \(k\) increases (see Fig.~\ref{fig:stiffness}).

Actually, for technical purposes, and following ideas that are very standard in the field of conservation laws, we first consider the parabolic regularization of \eqref{eq:CL} 

\begin{equation}
\label{eq:RegPara}
\left\{\begin{aligned}
&\partial_t \rho + \nabla \cdot (F_k(\rho)U) = \eps \Delta \rho, \\
&0 \leq \rho(t=0) \leq 1, \, \int \rho(t=0)  =1.
\end{aligned}\right.
\end{equation}
Doing so, we do not have to worry about the regularity of the solutions. We consider $\rho_{k,\eps}$ the unique smooth solution of \eqref{eq:RegPara}, and define $p_{k,\eps} = \rho_{k,\eps}^k$. We prove that in, topologies to be precised further,

\[
(\rho_{k,\eps},p_{k,\eps}) {\longrightarrow} (\rho,p),  \qquad \text{as} \quad k \to + \infty, \; \eps  \to 0.
\]
Due to the nonlinearities and the singular limit,  we need  specific compactness properties to prove that $(\rho,p)$ satisfies \eqref{eq:FBF} in some sense. Also, to ensure that \((\rho,p)\) satisfies the continuity equation (\ref{eq:continuity}) together with the law of state (\ref{eq:state}), we need to pass to the limit in a product term, which requires strong compactness for at least one of the sequences \(\rho_k\) or \(p_k\).

The complementarity equation (\ref{eq:complementarity}) satisfied by \(p\) is formally contained in the continuity equation (\ref{eq:continuity}) whenever \(\Omega(t)\) in \eqref {def:Omega} is smooth. However, the boundary condition (\ref{eq:boundary}) forces us to carefully address the regularity of \(p\). As shown in Figure.~\ref{fig:FBF1d}, we  expect discontinuities, in space and time, for the pressure $p$. 
We can prove that $p$ has some \textit{directional regularity} in space : in the direction of $U$, it can jump \textit{upwards}, but not \textit{downwards}. Since $p$ vanishes outside of $\Omega(t)$, this gives a meaning to the boundary condition \eqref{eq:boundary}.

\begin{figure}[t]
\centering
\begin{subfigure}{0.45\textwidth}
\centering
\includegraphics[width = \linewidth]{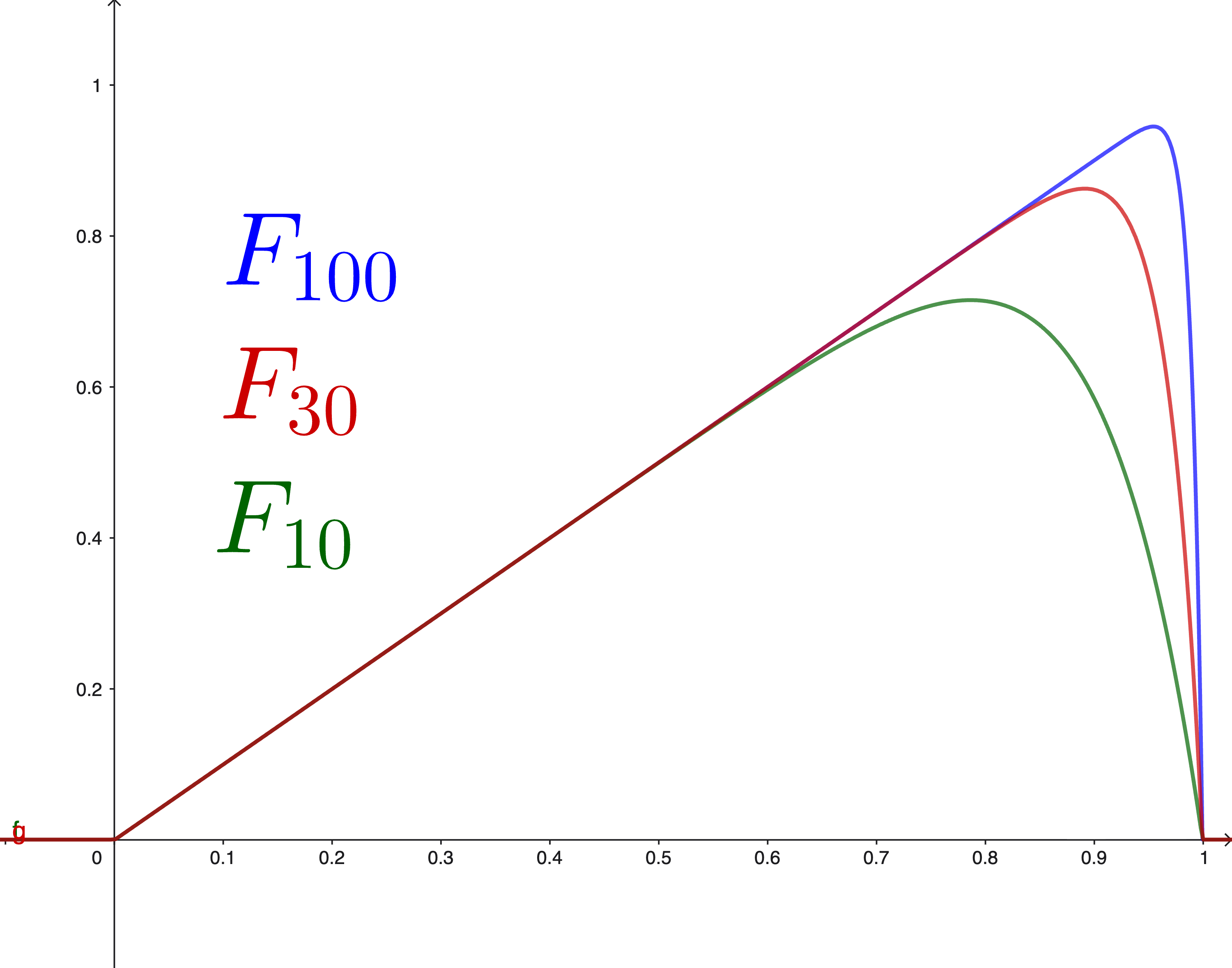}
\caption{Flux function $F_k$}
\label{fig:flux}
\end{subfigure}
\begin{subfigure}{0.45\textwidth}
\centering
\includegraphics[width = \linewidth]{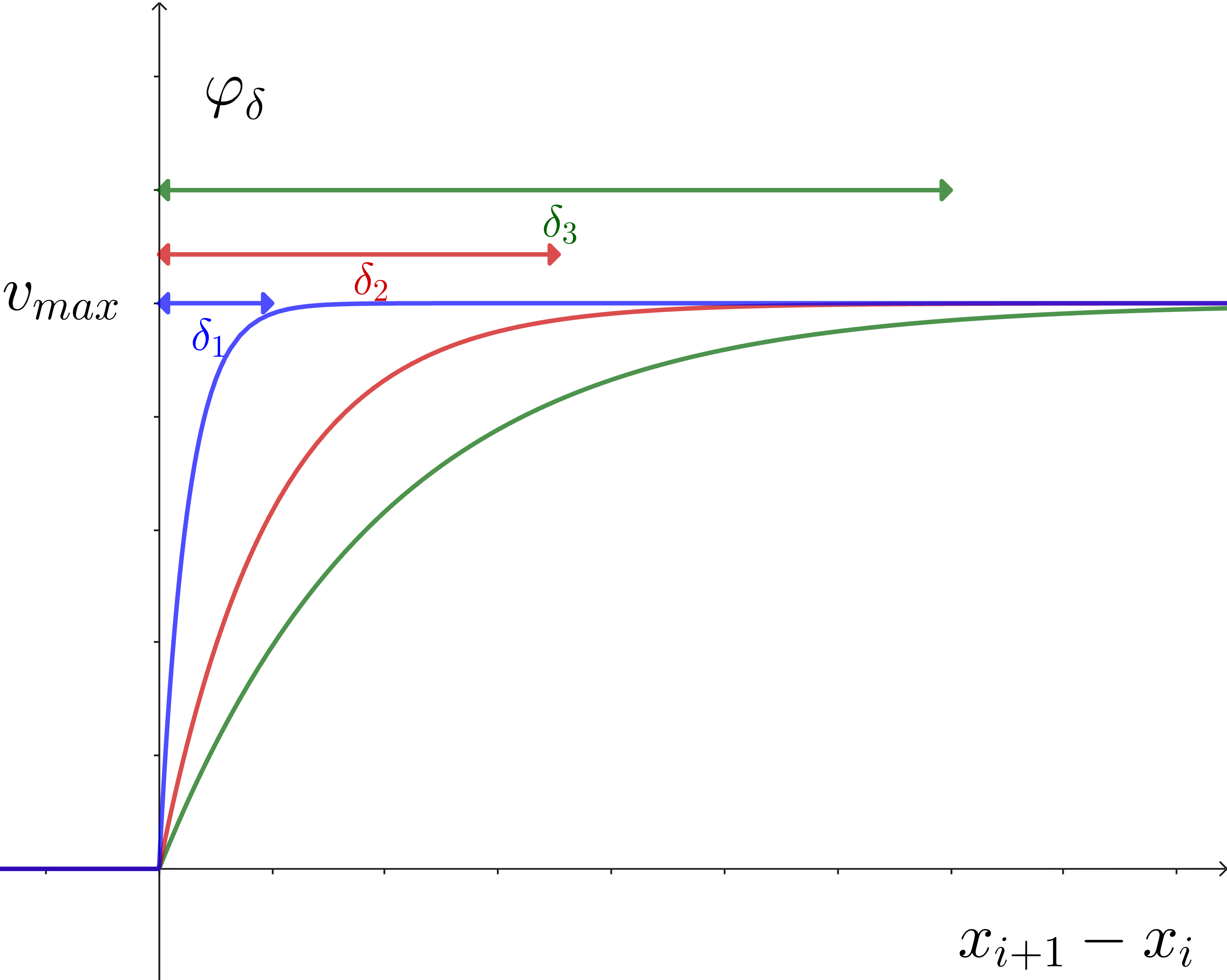}
\caption{FTL}
\label{fig:FTL}
\end{subfigure}
\caption{Corresponding approached models in the macroscopic and microscopic setting}
\label{fig:stiffness}
\end{figure}

Overall, the main goal of the paper is to study a singular limit of scalar conservation laws as a parameter tends to infinity. This type of asymptotics in the so-called \textit{incompressible limit}, has been investigated for instance in \cite{david_uniform_2024,perthame_hele-shaw_2014,MR4952923} for hard constraints models derived from the porous media equation. The main difference with the equations considered in these works is that the pressure appears as a gradient, and methods have been developed to prove strong compactness properties that seem, at present, out of reach in our setting.
\\

\textbf{Organization of the paper}. In section~\ref{sec:main} we show BV estimates for the solution $\rho$ of \eqref{eq:CL} that are uniform in time and $k$. These estimates on $\rho$ provide strong compactness, which is sufficient to establish (\ref{eq:continuity}) and (\ref{eq:state}).
The same analysis also provides very partial, but uniform in $k$, BV estimates in space for the pressure $p$; because the equation on $p$ has no time derivative, these estimates do not give strong convergence of the pressure in $L^1$. Then, we explain how to handle (\ref{eq:boundary}) using only weak convergence arguments. In section~\ref{sec:entropies}, we detail some remarks about the qualitative behavior of the solutions, and establish the entropy inequalities for \eqref{eq:CL}, before passing these inequalities to the limit $k \to + \infty$ and discuss the issue of uniqueness in \eqref{eq:FBF}.

\section{Main results}
\label{sec:main}

We begin by introducing some notations used throughout the paper, and state some preliminary results from parabolic theory.

\subsection{Notations and preliminary results}

In the rest of the paper, we use the following notations and assumptions:

\begin{itemize}[label=$\bullet$]

\item $d \in \N^*$ the dimension,
\item $k > 0$ the stiffness parameter,
\item $U\in W^{2,\infty}(\R^d),$
\item $T > 0$, and $Q_T = [0,T] \times \R^d$,
\item $F_k(\rho) = \rho(1-\rho^k)$ the flux function,
\item $p_k = \rho^k$ the pressure,
\item $\Sigma(t) = \{x\in \R^d, \quad\rho(t,x) = 1\}$ the saturated zone,
\item $\Omega(t) = \overset{\circ}{\Sigma}(t)$ its interior,
\item $\p \Omega_f(t) = \left\{ x \in \p \Omega(t), \; n(x)\cdot U(x) > 0 \right \}$, where $n(x)$ is the outward normal vector in $x \in \p \Omega(t)$,
\item $B_R(\R^d)$ the ball of center $0$ and radius $R >0$ in $\R^d$,
\item $\eta_R \geq 0$ a cutoff function, equal to $1$ on $B_R(\R^d)$, with exponential decay at $+ \infty$, and such that $\| \nabla \eta_R \| \le \frac{C}{R}\eta_R$ and $| \Delta \eta_R| \le \frac{C}{R^2}\eta_R$,
\item the $\sgn$ the function
\begin{align*}
\sgn(\rho) = \left\{
\begin{array}{ll}
1 & \text{for } \rho > 0, \\
0 & \text{for } \rho = 0, \\
-1 \; & \text{for } \rho < 0,
\end{array}
\right.
\end{align*}
\item $\rho^0 \in L^1(\R^d) \cap L^\infty(\R^d)\cap BV(\R^d)$ the initial condition.

\end{itemize}
We consider the following parabolic PDE on $Q_T$ :

\begin{equation}
\label{eq:RegParaBis}
\left\{\begin{aligned}
&\partial_t \rho + \nabla \cdot (F_k(\rho)U) = \eps \Delta \rho, \\
&0 \leq \rho(t=0) = \rho^0 \leq 1, \, \int \rho^0  =1.
\end{aligned}\right.
\end{equation}

We state without proof the following preliminary result that comes from parabolic regularity theory. For an extensive presentation of the framework and the main theorems, we refer to \cite{ladyzhenskaia_linear_1968, lieberman_second_2005}. For statements more tailored for parabolic perturbation of conservation laws, see \cite{malek_weak_1996}. The results presented in the latter are proven for a homogeneous flux, but these can be easily adapted to the case of a flux with Lipschitz dependence in space.

\begin{lemma}
\label{lm:Prel}
The PDE (\ref{eq:RegParaBis}) has a unique solution $\rho$ in $L^\infty(Q_T)\cap L^1(Q_T)$, smooth on $(0,T] \times \R^d$. Moreover, this equation enjoys a standard $L^1$-contraction property, together with a comparison principle. In particular, we have, for every $t \in [0,T]$
\begin{align*}
0 \le \rho(t) &\le 1, \\
\| \rho(t) \|_{L^1(\R^d)} &= \| \rho^0 \|_{L^1(\R^d)} = 1.
\end{align*}
We also have from parabolic regularity theory the following regularity, for all $t>0$, 
\begin{align*}
\rho(t) \in W^{3,\infty}(\R^d)\cap W^{3,1}(\R^d).
\end{align*}
\end{lemma}

\subsection{BV estimates for the conservation law}

Despite being covered by Lemma~\ref{thm:LemmaMultiD} in several dimensions without additional assumptions, we first state and prove our main result in one dimension for simplicity. We provide detailed arguments using cutoff functions for readers who may be unfamiliar with these techniques; this allows us to omit such details in the multidimensional case and thereby improve clarity.

\begin{lemma}[BV estimates in $1D$]
\label{Lemma1D}
Assume $d = 1$, $\p_xU \le0$, and let $\rho_{k,\eps}$ be the unique solution of (\ref{eq:RegParaBis}), which we denote by $\rho$ for simplicity. Then, the following estimates hold true, uniformly in $k$ and $\eps$, for all $0 \le t \le T$:

\[
\| \p_x \rho(t)\|_{L^1(\R)} \leq C(T)\|\rho^0\|_{BV}, \qquad \int_0^T \int_{\R} \big|\frac{\p \rho^{k+1}}{\p x }(t) \big|\;   \big|{\partial_x} U\big| \leq C(T)\|\rho^0\|_{BV}.
\]
\end{lemma}

\begin{proof}
Differentiating equation \eqref{eq:RegParaBis} with respect to $x$, we obtain
\begin{align*}
\partial_t \partial_x \rho + \partial_x(F'_k(\rho)\, \p_x\rho \, U) + \partial_x(F_k(\rho)\partial_x U) = \eps \Delta \p_x\rho.
\end{align*}
Multiplying by $\sgn(\partial_x \rho)$, and using the standard chain rule in Sobolev spaces, we get
\begin{align}
\label{eq:gradient_1D}
\p_t |\p_x \rho | + \p_x(F_k'(\rho)|\p_x \rho |U) + F_k'(\rho)|\p_x \rho |\p_x U + F_k(\rho)\sgn(\p_x\rho) \p_{xx} U \le \eps \Delta | \p_x\rho |.
\end{align}
Since $ \rho (t)\in W^{2,\infty}(\R^d)$, we can multiply by $\eta_R$ and integrate over the whole space. After integrating by parts the second and the last terms, we obtain
\begin{align*}
\frac{d}{dt} \int_{\R} |\p_x \rho |\eta_R \le \underbrace{\int_{\R^d} F_k'(\rho) |\p_x \rho | U \p_x \eta_R}_{\rm I}  - \underbrace{\int_{\R^d} F_k'(\rho)|\p_x \rho | \p_x U \eta_R}_{\rm II} - \underbrace{\int_{\R^d} F_k(\rho) \sgn(\p_x\rho) \p_{xx}U \eta_R}_{\rm III} \\
+ \underbrace{\eps \int_{\R^d} |\p_x \rho | \p_{xx} \eta_R}_{\rm IV}.
\end{align*}
Since $0 \le \rho \le 1$ and $|F_k(\rho)| \le |\rho|$, and since $\| \rho (t) \|_{L^1(\R)} = \| \rho^0 \|_{L^1(\R)}=1$, the third term on the right-hand side (\rm III) is bounded by a constant independent of $k$, $\eps$, and $R$. Keeping in mind that $\p_x U \le 0$, it only remains to bound $F'_k(\rho) = 1-(k+1)\rho^k$ from above in the second term on the right-hand side (\rm II). For the first term (\rm I), we control $|F_k'(\rho)|$ by $\sup_{\rho \in [0,1]} |F_k'(\rho)|= k$, and use the properties of the cutoff function to estimate $\p_x \eta_R$, as well as $\p_{xx}\eta_R$ in the last term (\rm IV). We finally obtain, with $C$ independent of $k$ and $\eps$,
\begin{align*}
\frac{d}{dt} \int_{\R} |\p_x \rho |\eta_R \le \|\rho^0\|_{L^1}\|\p_{xx}U\|_{L^\infty} + \left(\|\p_xU\|_{L^\infty}+\frac{Ck\|U\|_{L^\infty}}{R}+\frac{C\eps}{R^2}\right) \int_{\R} |\p_x \rho | \eta_R.
\end{align*}
We recall that the Gronwall lemma provides $y(t)\leq \left(y(0)+\frac AB\right)e^{Bt}$ whenever $y$ satisfies $y'(t)\le A+By(t)$. Here we obtain, with $A=\|\rho^0\|_{L^1}\|\p_{xx}U\|_{L^\infty}$ and $B=\|\p_xU\|_{L^\infty}+\frac{Ck\|U\|_{L^\infty}}{R}+\frac{C\eps}{R^2}\geq \|\p_xU\|_{L^\infty}$:
\begin{align*}
\int_{\R} |\p_x \rho |(t) \eta_R \le \left(\|\p_x\rho^0\|_{L^1}+\frac{\|\rho^0\|_{L^1}\|\p_{xx}U\|_{L^\infty}}{\|\p_x U \|_{L^\infty}}\right)\exp \left( \left(C + \frac{Ck}{R}+\frac{C\eps}{R^2}\right)T \right),
\end{align*}
where $\|\p_x \rho^0\|_{L^1}$ is understood as the total variation if $\rho^0$ is of bounded variation. Letting $R \to + \infty$, we finally obtain
\begin{align*}
\left \| \p_x \rho (t) \right\|_{L^1(\R)} \le C(T)\|\rho^0\|_{BV(\R)},
\end{align*}
where $C(T)$ depends only on the $W^{2,\infty}$ norm of $U$.

The use of cutoff functions $\eta_R$ is only meant to justify rigorously the integration of $\p_x(F_k'(\rho)|\p_x\rho|U)$ and $\Delta |\p_x\rho |$ over the whole space, yielding $0$ as expected. We shall omit this argument in the rest of the paper. 

To obtain the estimate on the pressure, we integrate inequality \eqref{eq:gradient_1D} over space and time, and obtain
\begin{align*}
\| \p_x \rho (T) \|_{L^1(\R)} - \| \p_x \rho^0 \|_{L^1(\R)} + \int_0^T \int_{\R}(1-(k+1)\rho^k)|\p_x \rho | \p_x U + \int_0^T \int_{\R} F_k(\rho)\sgn(\p_x\rho)\p_{xx}U \le 0,
\end{align*}
which implies, since $\p_x U \le 0$,
\begin{align*}
\int_0^T\int_\R \big|\frac{\p \rho^{k+1}}{\p x } \big|\;   \big|{\partial_x} U\big| \le \int_0^T\int_{\R} | \p_x \rho | | \p_x U | + T \| \rho^0 \|_{L^1(\R)}\|U\|_{W^{2,\infty}} + \| \p_x \rho^0 \|_{L^1(\R)} - \| \p_x \rho (T) \|_{L^1(\R)} \\
\le C(T)\|\rho^0\|_{BV(\R)}.
\end{align*}
\end{proof}

Now we state the same lemma in several dimensions, using similar proof ideas but with heavier computations. 

\begin{lemma}[BV estimates in d dimensions]
\label{thm:LemmaMultiD}
Let $\rho$ be the unique solution of (\ref{eq:RegParaBis}), and denote by $D^SU$ the symmetric part of the jacobian matrix of $U$. Suppose that $D^SU \leq- \alpha I_d$ for some $\alpha\ge0$. Then for the Euclidian norm in $\R^d$, and for all $T>0$,  we have
\begin{align} \label{eq:BV} 
\int_{\R^d}\| \nabla \rho(t)\| \leq C(T)\|\rho^0\|_{BV}, \; \forall t \leq T, \qquad \alpha \int_0^T \int_{\R^d}\| \nabla \rho^{k+1} (t)\| dt\leq C(T)\|\rho^0\|_{BV}.
\end{align}
\end{lemma}

\begin{proof}
We proceed as before, although the computations are more involved. Let $H : \R^d \rightarrow \R$ be a convex function (which will then be chosen to approximate the Euclidean norm), and write $H_i$ for its spatial partial derivatives. Differentiating equation (\ref{eq:RegParaBis}) with respect to $x_i$, multiplying by $H_i(\nabla \rho)$, and summing over $i$, we obtain
\begin{align}
\label{eq:ptH}
\p_t H(\nabla \rho) + \sum_{i,j} \p_{x_i,x_j}\left(F_k(\rho)U^j\right) H_i(\nabla \rho)
= \eps \sum_i H_i(\nabla \rho) \Delta \p_{x_i}\rho = \eps \nabla H(\nabla \rho) \cdot \nabla (\Delta \rho).
\end{align}

We now analyze the right-hand side. We compute
\begin{align*}
\Delta \big( H(\nabla \rho) \big)
&= \sum_i \p_{x_i} \left( \sum_j H_j(\nabla \rho) \p_{x_i x_j} \rho \right) \\
&= \sum_{i,j} \p_{x_i}\big(H_j(\nabla \rho)\big) \p_{x_i x_j} \rho
  + \sum_{i,j} H_j(\nabla \rho) \p_{x_j} \p_{x_i x_i} \rho \\
&= \sum_{i,j,k} H_{k,j}(\nabla \rho) \p_{x_i x_k}\rho \, \p_{x_i x_j}\rho
  + \sum_i H_i(\nabla \rho)\p_{x_i}\Delta \rho.
\end{align*}
This can be written more concisely as
\begin{align*}
\Delta \big( H(\nabla \rho) \big)
= \mathrm{Tr}\!\left(\nabla^2 \rho \cdot \nabla^2 H(\nabla \rho) \cdot \nabla^2 \rho\right)
+ \nabla H(\nabla \rho) \cdot \nabla (\Delta \rho).
\end{align*}
The second term on the right-hand side coincides with the last term in (\ref{eq:ptH}). Since $H$ is convex, we have
\begin{align*}
\mathrm{Tr}\!\left(\nabla^2 \rho \cdot \nabla^2 H(\nabla \rho) \cdot \nabla^2 \rho\right) \ge 0,
\end{align*}
and therefore
\begin{align*}
\nabla H(\nabla \rho) \cdot \nabla (\Delta \rho)
\le \Delta \big( H(\nabla \rho) \big).
\end{align*}

Applying this inequality and integrating over $\R^d$—see the one-dimensional proof for the treatment of distributional terms using cutoff functions—we obtain from \eqref{eq:ptH}
\begin{align*}
\frac{d}{dt}\int_{\R^d} H(\nabla \rho)
+ \sum_{i,j} \int_{\R^d} \p_{x_i,x_j}\left(F_k(\rho)U^j\right) H_i(\nabla \rho)
\le 0.
\end{align*}
Integrating by parts yields
\begin{align*}
\frac{d}{dt}\int_{\R^d} H(\nabla \rho)
- \sum_{i,j} \int_{\R^d} \p_{x_i}\left(F_k(\rho)U^j\right)
  \p_{x_j}\left(H_i(\nabla \rho)\right)
\le 0.
\end{align*}

The second term on the left-hand side can be decomposed as
\[
\begin{aligned}
\sum_{i,j}\int_{\R^d} \p_{x_i}(F_k(\rho)U^j)\p_{x_j}(H_i(\nabla \rho))
&= \sum_{i,j}\int_{\R^d} \p_{x_i}(F_k(\rho)U^j) \p_{x_j}(H_i(\nabla \rho)) \\
&= \sum_{i,j,l} \int_{\R^d} F'_k(\rho) U^j H_{i,l} \rho_i \rho_{l,j}
 + \sum_{i,j} \int_{\R^d} F_k(\rho) U^j_i \p_{x_j}(H_i(\nabla \rho)).
\end{aligned}
\]

Integrating by parts once more in the second term, we obtain
\begin{align*}
\sum_{i,j} \int_{\R^d} F_k(\rho) U^j_i \p_{x_j}(H_i(\nabla \rho))
&= - \sum_{i,j} \int_{\R^d} \p_{x_j}\big(F_k(\rho) U^j_i\big) H_i(\nabla \rho) \\
&= - \sum_{i,j} \int_{\R^d} F'_k(\rho) \rho_j U^j_i H_i(\nabla \rho)
   - \sum_{i,j} \int_{\R^d} F_k(\rho) U^j_{i,j} H_i(\nabla \rho).
\end{align*}

Altogether, this gives
\begin{align*}
\frac{d}{dt}\int_{\R^d} H(\nabla \rho)
&\le \sum_{i,j,l} \int_{\R^d} F'_k(\rho) U^j H_{i,l} \rho_i \rho_{l,j}
 - \sum_{i,j} \int_{\R^d} F'_k(\rho) \rho_j U^j_i H_i(\nabla \rho) \\
&\quad - \sum_{i,j} \int_{\R^d} F_k(\rho) U^j_{i,j} H_i(\nabla \rho).
\end{align*}
In a more compact form, we may write
\begin{align*}
\frac{d}{dt}\int_{\R^d} H(\nabla \rho)
&\le \int_{\R^d} F_k'(\rho)
   \left\langle \nabla^2 H(\nabla \rho)\nabla \rho,
   \nabla^2 \rho\, U \right\rangle \\
&\quad - \int_{\R^d} F_k'(\rho)
   \left\langle DU \nabla H(\nabla \rho), \nabla \rho \right\rangle \\
&\quad - \int_{\R^d} F_k(\rho)
   \left\langle \nabla(\nabla \cdot U), \nabla H(\nabla \rho) \right\rangle.
\end{align*}

We now choose a function $H$ which approximates $H(q) = \|q\|$ (we omit the details on the approximation), for which we have
\begin{align*}
\nabla H(q) = \frac{q}{\|q\|},
\qquad
H_i(q) = \frac{q_i}{\|q\|},
\qquad
\nabla^2 H(q) = \frac{I_d}{\|q\|} - \frac{q \otimes q}{\|q\|^2}.
\end{align*}
This gives us
\[
\nabla^2 H(q) q = 0,
\]
and therefore
\[
\int_{\R^d} F_k'(\rho)
\left\langle \nabla^2 H(\nabla \rho)\nabla \rho,
\nabla^2 \rho\, U \right\rangle
= 0.
\]
Moreover, $H$ is radial, so $\nabla H$ is parallel to its argument, and

\[
\begin{aligned}
\left\langle DU \nabla H(\nabla \rho),\nabla \rho \right\rangle &= \left\langle D^SU \nabla H(\nabla \rho),\nabla \rho \right\rangle \\
&= \left\langle D^SU \f{\nabla \rho}{\| \nabla \rho \|},\nabla \rho \right\rangle.
\end{aligned}
\]
Hence,
\begin{equation}
\label{eq:GronwallMultiD}
\frac{d}{dt}\int_{\R^d} \|\nabla \rho\|
\le - \int_{\R^d} F_k'(\rho)
   \left\langle D^S U \frac{\nabla \rho}{\|\nabla \rho\|},
   \nabla \rho \right\rangle
 - \int_{\R^d} F_k(\rho)
   \left\langle \nabla(\nabla \cdot U),
   \frac{\nabla \rho}{\|\nabla \rho\|} \right\rangle.
\end{equation}

On the right-hand side, the first term is controlled using $-\beta \le D^S U \le 0$, for $\beta = \|U\|_{W^{2,\infty}}$, so that $F_k'$ only needs to be bounded from above (uniformly in $k$). The third term is estimated by the $L^1$ norm of $\rho$ and the $W^{2,\infty}$ norm of $U$. We thus obtain
\begin{align*}
\frac{d}{dt}\int_{\R^d} \|\nabla \rho(t)\|
\le \beta \int_{\R^d} \|\nabla \rho(t)\|
+ \underbrace{\|\rho(t)\|_{L^1(\R^d)}}_{=\|\rho^0\|}
  \|U\|_{W^{2,\infty}}.
\end{align*}
Arguing as before, this yields the desired BV bound on $\rho$. 
\\

For the estimate on the pressure, we go back to \eqref{eq:GronwallMultiD}, integrate in time and remind that $F_k'(\rho) = 1-(k+1)\rho^k$, which means that

\[
\|\nabla \rho(T)\|_{L^1(\R^d)^d}-\|\nabla \rho^0\|_{L^1(\R^d)^d}
\le - \int_0^T\int_{\R^d} (1-(k+1)\rho^k)
   \left\langle D^S U \frac{\nabla \rho}{\|\nabla \rho\|},
   \nabla \rho \right\rangle
 + T\|U\|_{W^{2,\infty}},
\]
which rewrites

\[
\begin{aligned}
-\int_0^T\int_{\R^d}(k+1)\rho^k \left\langle D^S U \frac{\nabla \rho}{\|\nabla \rho\|}, \nabla \rho \right\rangle \le \|\nabla \rho^0\|_{L^1(\R^d)^d} -\|\nabla \rho(T)\|_{L^1(\R^d)^d} + T \|U\|_{W^{2,\infty}} \\
-\underbrace{\int_0^T\int_{\R^d}\left\langle D^S U \frac{\nabla \rho}{\|\nabla \rho\|},\nabla \rho \right\rangle}_{\le \beta \|\nabla \rho \|_{L^1([0,T];L^1(\R^d)^d)}}.
\end{aligned}
\]
Notice that

\[
-(k+1)\rho^k \left\langle D^S U \frac{\nabla \rho}{\|\nabla \rho\|}, \nabla \rho \right\rangle  \ge \alpha (k+1)\rho^k\|\nabla\rho\| = \alpha \|\nabla \rho^{k+1}\|,
\]
which implies that 

\[
\begin{aligned}
\alpha \int_0^T\|\nabla \rho^{k+1}(t)\|dt &\le \|\nabla \rho^0\|_{L^1(\R^d)^d} -\|\nabla \rho(T)\|_{L^1(\R^d)^d} + T \|U\|_{W^{2,\infty}} + \beta \|\nabla \rho \|_{L^1([0,T];L^1(\R^d)^d)}\\
&\le C(T)\|\rho^0\|_{BV}
\end{aligned}
\]
and Lemma \ref{thm:LemmaMultiD} is proved. 
\end{proof}

This lemma provides strong compactness for $\rho_{k,\eps}$ as we state it now.
\begin{corollary}[Strong convergence of $\rho_{k,\eps}$]
\label{thm:conv_rho}
Up to a subsequence, for all $t \in [0,T]$, we have
\begin{align*}
    \rho_{k,\eps}(t) \longrightarrow \rho(t),\quad \text{in} \quad L^1_{loc}(\R^d), \qquad \text{as } \quad
    \eps \to 0, \; k \to \infty.
\end{align*}
\end{corollary}

\begin{proof}
Observe that the $W^{-1,1}(\R^d)$ norm of $\p_t \rho_{k,\eps}(t)$ is bounded for every $t>0$, since
\[
\begin{aligned}
\|\p_t \rho_{k,\eps}(t)\|_{W^{-1,1}(\R^d)}
&\le \|\nabla \cdot (F_k(\rho_{k,\eps})U)(t)\|_{W^{-1,1}(\R^d)}
   + \eps \|\Delta \rho_{k,\eps}(t)\|_{W^{-1,1}(\R^d)} \\
&\lesssim \|F_k(\rho_{k,\eps})U\|_{L^1(\R^d)^d}
   + \eps \|\nabla \rho_{k,\eps}(t)\|_{L^1(\R^d)^d} \\
&\le C(T)\|\rho^0\|_{BV},
\end{aligned}
\]
where $C(T)$ is independent of $k$ and $\eps$, provided that $\eps \le 1$. 

Combining this estimate with the previous bounds, we obtain
\[
\rho_{k,\eps} \text{ bounded in } L^{\infty}([0,T],BV(\R^d)),
\qquad
\partial_t \rho_{k,\eps} \text{ bounded in } L^{\infty}([0,T],W^{-1,1}(\R^d)).
\]
Since, by the Rellich theorem, $BV(\R^d)$ is compactly embedded in $L^1_{loc}(\R^d)$, and $L^1_{loc}(\R^d)$ is continuously embedded in $W^{-1,1}(\R^d)$, the Aubin--Lions--Simon lemma~\cite{simon_compact_1986} implies that, up to a subsequence,
\[
\rho_{k,\eps} \to \rho \quad \text{in } C([0,T],L^1_{loc}(\R^d)),
\]
which concludes the proof.
\end{proof}

Now, we are able to prove that, at the limit $\eps \to 0$ and $k \to + \infty$, equations \eqref{eq:continuity} and \eqref{eq:state} are satisfied.

\begin{corollary}
\label{thm:state&cont}
    Denote $\rho$ the $C([0,T],L^1_{loc})$ limit of $\rho_{k,\eps}$ after extraction, and $p$ the $L^\infty$-weak $\ast$ limit of $p_{k,\eps}$ after further extraction. Then $(\rho,p)$ satisfies both the continuity equation \eqref{eq:continuity} in the distributional sense, and the law of state \eqref{eq:state}.
\end{corollary}

\begin{proof}
    For the continuity equation \eqref{eq:continuity}, since $p_{k,\eps} = \rho_{k,\eps}^k$ is bounded in $L^{\infty}(Q_T)$, one can extract a subsequence such that 
    \[
    p_{k,\eps} \overset{\ast}{\rightharpoonup} p 
\quad \text{in} \quad L^{\infty}(Q_T).
    \]
    Therefore, by weak-strong limits 

    \[
    \rho_{k,\eps}(1-p_{k,\eps}) \overset{\ast}{\rightharpoonup} \rho(1-p) 
\quad \text{in} \quad L^{\infty}(Q_T),
    \]
    and thus 

    \[
    \p_t \rho_{k,\eps} + \nabla \cdot (\rho_{k,\eps}(1-p_{k,\eps})U) \overset{\mathcal{D}'}{\longrightarrow} \p_t \rho +\nabla \cdot (\rho(1-p)U).
    \]
    Since 
    \[
    \eps \Delta \rho_{k,\eps} \overset{\mathcal{D}'}{\longrightarrow} 0,
    \]
    we have that $(\rho,p)$ satisfy \eqref{eq:continuity} in the distributional sense. As for the law of state \eqref{eq:state}, we also have by weak-strong limits that 

    \[
    p_{k,\eps}(1-\rho_{k,\eps}) \overset{\mathcal{D}'}{\longrightarrow} p(1-\rho).
    \]
    But notice that 

    \[
    \begin{aligned}
    \|p_{k,\eps}(1-\rho_{k,\eps})\|_{L^{\infty}(Q_T)} &= \|\rho_{k,\eps}^k(1-\rho_{k,\eps})\|_{L^{\infty}(Q_T)}\\
    &\le \f{1}{k+1}\\
    &\longrightarrow 0.
    \end{aligned}
    \]
    So, by uniqueness of the limit in $\mathcal{D}'$
    \[
    p(1-\rho) = 0.
    \]
\end{proof}

 The major difficulty is that such strong compactness does not hold for the pressure $p_{k,\eps}$. The estimates and compactness results  at hand  are sufficient to pass to the limit in the continuity equation~(\ref{eq:continuity}) and in the law of state~(\ref{eq:state}), showing that the limits $\rho$ and $p$ satisfy both relations. However, we still need some information on the convergence of $p$ in order to prove that the pressure vanishes on the frontal part $\partial \Omega_f$. The difficulties concerning the compactness of $p_{k,\eps}$ are multiple
\begin{itemize}
\item First, the BV estimate on $\rho_{k,\eps}^{k+1}$ only holds when $\alpha>0$;
\item Second, this is not a bound on the BV norm of $p_{k,\eps}$ but on that of $p_{k,\eps}^{1+1/k}$ or, equivalently, of $\rho_{k,\eps} p_{k,\eps}$;
\item Last, the estimate only yields spatial compactness integrated in time, while the lack of time regularity for the pressure prevents the use of the Aubin--Lions lemma.
\end{itemize}
The first drawback is not a big deal as it is enough to consider a slightly more restrictive assumption, i.e., $\alpha>0$ instead of $\alpha\ge 0$. The second can be treated as in Corollary~\ref{thm:state&cont}, using weak-strong limits and the identity $\rho p = p$ Unfortunately, the third drawback is a serious issue and in the next section we develop a different argument to study the limit $p$ and establish the boundary condition~(\ref{eq:boundary}) for the pressure.

\begin{remark}
Notice that $\rho_{k,\eps}$ can be interpreted as a solution of the continuity equation with velocity
\[
v_t = (1-\rho^k)U - \eps\frac{\nabla \rho}{\rho}.
\]
In particular, $\rho_{k,\eps}$ is an absolutely continuous curve of probability measures with respect to the $W_1$ metric (see~\cite{santambrogio_optimal_2015} for the definition of $W_1$). Its metric derivative is given by the $L^1(\rho)$ norm of $v_t$ :
\[
\begin{aligned}
\int_{\R^d} \left\|(1-\rho^k(t))U-\eps\f{\nabla \rho(t)}{\rho(t)} \right\| \rho(t) &\le \int_{\R^d} \rho(t)(1-\rho^k(t))\|U\| + \eps \int_{\R^d}\|\nabla \rho(t)\| \\
&\le \underbrace{\|\rho(t)\|_{L^1(\R^d)}}_{=\|\rho^0\|_{L^1(\R^d)}}\|U\|_{\infty} + \eps C(T)\|\rho^0\|_{BV},
\end{aligned}
\]
 which is uniformly bounded thanks to the previous lemma (see~\cite{ambrosio_gradient_2005} for the notion of metric derivative and its relation to absolutely continuous curves of measures solving the continuity equation). 

This uniform bound in $k$ and $\eps$ implies that $\rho_{k,\eps}$ are equi-Lipschitz curves valued into the Wasserstein space $W_1$. To use the Ascoli-Arzela theorem, we need compactness of $\rho_{k,\eps}(t)$ for each $t$ in $W_1$, which can be obtained with a uniform control over second order moments. Set $M(t) = \int_{\R^d}|x|^2\rho_{k,\eps}(t)$ and compute 

\[
\begin{aligned}
\f{d}{dt} M(t) 
&=  - \int_{\R^d} |x|^2\nabla \cdot(F_k(\rho_{k,\eps}(t))U)+ \eps \int_{\R^d} |x|^2\Delta \rho_{k,\eps}(t)\\
&= 2 \int_{\R^d} F_k(\rho_{k,\eps}(t))U\cdot x + 2d\eps\underbrace{\int_{\R^d}\rho_{k,\eps}}_{=1}\\
&\le \|U\|_{L^\infty(\R^d)}\int_{\R^d}|x|\rho_{k,\eps}(t) +2d\eps \\
&\le  \|U\|_{L^\infty(\R^d)} \left(\int_{\R^d}|x|^2 \rho_{k,\eps}(t)\right)^{1/2} \underbrace{\left(\int_{\R^d}\rho_{k,\eps}(t)\right)^{1/2}}_{=1} + 2d\eps\\
&\le  \|U\|_{L^\infty(\R^d)} M(t)^{1/2} + 2d\eps \\
&\le CM(t) + C,
\end{aligned}
\]
provided that $\eps \le 1$. Using the Gronwall lemma, assuming that 

\[
\int_{\R^d} |x|^2\rho^0 < + \infty,
\]
we get

\[
\int_{\R^d} |x|^2\rho_{k,\eps}(t) \le C(T) \int_{\R^d} |x|^2\rho^0+ C(T),
\]
where $C(T)$ does not depend on $k,\eps$.
Hence, by the Ascoli--Arzelà theorem, one can extract a subsequence such that, for every~$t$, $\rho_{k,\eps}(t)$ converges to a probability measure $\rho(t)$ with respect to the Wasserstein distance, that is, for the weak convergence of measures. Moreover, the bounds show that $\rho_{k,\eps}(t)$ is uniformly bounded in $BV(\R^d)$, which is compactly embedded in $L^1_{loc}(\R^d)$ and transforms the weak convergence into strong (actually in $L^1$ and not only $L^1_{loc}$ thanks to the aforementioned moment bounds).
\end{remark}

\subsection{The pressure vanishes on $\partial \Omega_f(t)$}

We first derive the equation satisfied by the pressure $p_{k,\eps}$ for finite $k$ and $\eps$, before carefully passing to the limit.

\begin{lemma}[Equation for the pressure for finite $k,\eps$]
Let $k,\eps >0$, let $\rho_{k,\eps}$ be the solution of \eqref{eq:RegParaBis}, and set $p_{k,\eps} = \rho_{k,\eps}^k$. Then $p_{k,\eps}$ satisfies the following equation, in the classical sense:
\begin{equation}
\label{eq:pressure_k}
\p_t p_{k,\eps} + \nabla\!\left[p_{k,\eps}-\f{k+1}{2}p_{k,\eps}^2\right]\!\cdot U 
+ k p_{k,\eps}(1-p_{k,\eps}) \nabla \cdot U 
= \eps \Delta p_{k,\eps} 
- 4\eps \left(1-\f1k\right)\!\left|\nabla p_{k,\eps}^{\f12}\right|^2.
\end{equation}
\end{lemma}

\begin{proof}
For simplicity, write $\rho=\rho_{k,\eps}$ and $p=p_{k,\eps}$. 
Multiplying equation \eqref{eq:RegParaBis} by $k\rho^{k-1}$, we obtain
\[
k\rho^{k-1}\p_t\rho 
+ k\rho^{k-1}(1-(k+1)\rho^k)\nabla \rho \cdot U 
+ k \rho^k(1-\rho^k)\nabla \cdot U 
= \eps k\rho^{k-1}\Delta \rho.
\]
This yields
\[
\p_t \rho^k 
+ \nabla \!\left[ \rho^k - \f{k+1}{2}\rho^{2k}\right]\!\cdot U 
+ k \rho^k(1-\rho^k) \nabla \cdot U 
= \eps \Delta \rho^k 
- \eps k(k-1)\underbrace{\rho^{k-2}|\nabla \rho |^2}_{\f{4}{k^2}|\nabla \rho ^{k/2}|^2},
\]
and \eqref{eq:pressure_k} follows immediately.
\end{proof}

Unfortunately, we do not obtain strong compactness for \( p_{k,\varepsilon} \), and therefore must proceed with care when passing to the limit in \eqref{eq:pressure_k}. To address this difficulty, we establish a more general identity similar to the previous one, where instead of considering $\rho_k^k$ we will look at $\rho_k^n$ for two independent indexes $k,n$. This formulation enables us to show that, in the limit, \( p \) satisfies a differential inequality, which in turn yields directional regularity.

\begin{lemma}
    Let $k,\eps >0$, let $\rho_{k,\eps}$ be the solution of \eqref{eq:RegParaBis}, and set $p_{k,\eps} = \rho_{k,\eps}^k$. Then, for every $n>0$, $p_{k,\eps}$ satisfies the following inequality, in the classical sense:

    \begin{equation}
    \label{eq:pressure_kn}
        \p_t \rho^n_{k,\eps} + \nabla[\rho^n_{k,\eps}-np_{k,\eps}]\cdot U + n[\rho^n_{k,\eps}-\rho_{k,\eps}^{n+k}](\nabla \cdot U) \le \eps \Delta\rho^n_{k,\eps}.
    \end{equation}
\end{lemma}

\begin{proof}
    Multiplying equation \eqref{eq:RegParaBis} by $n\rho_{k,\eps}^{n-1}$, we obtain

    \[
n \rho^{n-1}_{k,\eps}\p_t \rho + n\rho^{n-1}(1-(k+1)\rho^k_{k,\eps})\nabla \rho \cdot U + n\rho^n_{k,\eps}(1-\rho^k_{k,\eps})(\nabla \cdot U) = \eps n\rho^{n-1}_{k,\eps}\Delta \rho_{k,\eps} .
    \]
    This yields 
    
    \[
    \p_t \rho^n_{k,\eps} + \nabla\!\left[\rho^n_{k,\eps}-\f{n(k+1)}{n+k}\rho^{n+k}_{k,\eps} \right]\cdot U +n\rho^n_{k,\eps}(1-\rho^k_{k,\eps})(\nabla \cdot U) = \eps\Delta \rho^n_{k,\eps} - \underbrace{\eps n(n-1)\rho^{n-2}_{k,\eps}|\nabla \rho_{k,\eps}|^2}_{\ge 0},
    \]
    and thus 

    \[
    \p_t \rho_{k,\eps}^n + \nabla \!\left[\rho_{k,\eps}^n-\f{n(k+1)}{n+k}\rho_{k,\eps}^np_{k,\eps}\right]\cdot U +n(\rho^n_{k,\eps}-\rho_{k,\eps}^{n+k})(\nabla \cdot U) \le \eps\Delta\rho_{k,\eps}^n.
    \]
\end{proof}

\begin{lemma}[Limit inequality for the pressure]
\label{thm:compatibility}
Denote by $p$ the $L^\infty(Q_T)$-weak $\ast$ limit of $p_{k,\eps}$ obtained by further extraction. Then
\begin{equation}
    \label{eq:limit_pressure_ineq}
    -\nabla p \cdot U + \!\left[\mathbf{1}_{\{\rho = 1\}}-p\right]\nabla \cdot U \le 0.
\end{equation}
This implies that $p$ possesses \textbf{directional regularity} in space, in the sense that there exists a positive Radon measure $\mu$ on $Q_T$ and $g \in L^{\infty}(Q_T)$ nonnegative such that
\[
\nabla p \cdot U =\mu - g.
\]
\end{lemma}

\begin{proof}
Let $\rho_{k,\eps}$ solve \eqref{eq:RegParaBis}, and define $p_{k,\eps}=\rho_{k,\eps}^k$. Since $p_{k,\eps}$ is bounded in $L^\infty(Q_T)$, Corollary~\ref{thm:conv_rho} yields, up to a subsequence,
\begin{align*}
p_{k,\eps} &\overset{\ast}{\rightharpoonup} p 
\quad \text{in} \quad L^{\infty}(Q_T),\\
\rho_{k,\eps} &\to \rho 
\quad \text{in} \quad C([0,T],L^1_{loc}(\R^d)).
\end{align*}
Since $\rho_{k,\eps} \in[0,1]$, we also have that, as $k \to +\infty$, $\eps \to0$, 

\[
\rho_{k,\eps}^n {\longrightarrow} \rho^n \quad \text{in} \quad C([0,T],L^1_{loc}(\R^d)).
\]
By weak-strong limits, we get that 

\[
\rho_{k,\eps}^{n+k} = \rho_{k,\eps}^n p_{k,\eps} \underset{k \to + \infty,\eps \to 0}{\longrightarrow} \rho^np \quad \text{weak * in} \quad L^\infty(Q_T).
\]
However, since by Corollary \ref{thm:state&cont}, $\rho p=p$, we infer that 

\[
\rho^n p = p.
\]

Passing to the limit as $k\to+\infty$ and $\eps\to0$, we obtain in the sense of distributions that
\[
\p_t \rho^n + \nabla \!\left[\rho^n-np\right]\cdot U+n(\rho^n-p)(\nabla\cdot U) \le 0.
\]
Finally, dividing the last inequality by $n$ and taking $n \to +\infty$ gives, since $\rho^n \searrow \mathbf{1}_{\{\rho = 1\}}$

\[
-\nabla p \cdot U + [\mathbf{1}_{\{\rho=1\}}-p](\nabla \cdot U) \le0.
\]

As for the directional regularity of $p$, the previous inequality gives the existence of a non-negative Radon measure such that

\[
\nabla p\cdot U = \mu-g,
\]
where $g = -(\nabla \cdot U)[\mathbf{1}_{\{\rho=1\}}-p] \in L^\infty(Q_T)$.

\end{proof}

\begin{corollary}
\label{thm:boundary_condition}
Let $p$ be as in the previous lemma. Suppose that $-D^SU \ge \alpha I_d$ with $\alpha >0$, which implies that $p \in L^1([0,T];BV(\R^d))$.
If for all $t\in[0,T]$ the set $\Omega(t)$ is Lipschitz, then 
$p$ admits a well-defined trace on $\partial\Omega(t)$ 
for almost every $t$, and 

\[
p(t)=0 \quad \text{on} \quad\partial\Omega_f(t),
\]
for almost every $t$.
\end{corollary}

\begin{proof}
For almost every $t\in[0,T]$, $p(t)$ is a BV function. 
Hence it admits a well-defined trace on $\partial\Omega(t)$, denoted 
$p^{\Omega(t)}$. 
Since $p(t)$ vanishes in $\R^d\setminus\overline{\Omega(t)}$, 
its gradient as a measure satisfies
\[
\nabla p(t)
=
(\nabla p(t))\mathbf{1}_{\Omega(t)}\, d\mathcal{L}^d
-
p^{\Omega(t)} \, n_{\Omega(t)} 
\mathbf{1}_{\partial\Omega(t)} \mathcal{H}^{d-1}.
\]
See for instance Corollary 3.89 in \cite{ambrosio_functions_2000} for a complete proof, being careful about the fact that the convention in geometric measure theory is taking the \textit{inward normal}.
Taking the scalar product with $U$, we obtain
\[
\nabla p(t)\cdot U
=
(\nabla p(t)\cdot U)\mathbf{1}_{\Omega(t)}\, d\mathcal{L}^d
-
p^{\Omega(t)} (n_{\Omega(t)}\cdot U)
\mathbf{1}_{\partial\Omega(t)} \mathcal{H}^{d-1}.
\]
Since $\nabla p\cdot U = \mu-g$ and the measure $\mu$ is nonnegative, it follows that
$p^{\Omega(t)}$ must vanish on the set where 
$n_{\Omega(t)}\cdot U>0$, namely $\p \Omega_f(t)$.
\end{proof}

\subsection{Main theorem}

Altogether, we have proved our main result

\begin{theorem}[Stiff limit] 
\label{thm:main}
Assume that $D^SU \leq -\alpha Id$ for some $\alpha \geq 0$. There exists a couple $(\rho,p)$, with 
\[
\rho \in L^{\infty}([0,T], BV(\R^d)) \cap C([0,T],L^1_{loc}(\R^d)),
\qquad 
p \in L^{\infty}(Q_T),
\]
such that
\begin{equation}
\left\{
\begin{aligned}
 &\partial_t \rho + \nabla \cdot (\rho (1-p)U) = 0, \\
 &-\nabla p \cdot U +[\mathbf{1}_{\{\rho = 1\}}-p](\nabla \cdot U)  \le 0, \\
 &p(1-\rho) = 0, 
 \qquad 0 \le p \le 1, \qquad 0 \leq \rho \leq 1.
\end{aligned}
\right.
\end{equation}
\end{theorem}

\begin{remark}
Theorem~\ref{thm:main} shows that, denoting by  \(\rho_{k,\varepsilon}\) the solution of \eqref{eq:RegParaBis} and \(p_{k,\varepsilon} = \rho_{k,\varepsilon}^k\), then, up to the extraction of a subsequence, the limit pair \((\rho,p)\) satisfies \eqref{eq:continuity} and \eqref{eq:state}. Under the additional assumption that \(\alpha > 0\), it also follows—combined with Corollary~\ref{thm:boundary_condition}—that the boundary condition \eqref{eq:boundary} is satisfied. Consequently, we obtain almost all the ingredients required to establish the existence of a solution to the PDE system \eqref{eq:FBF} when \(\alpha > 0\), with the exception of \eqref{eq:complementarity}, which remains to be proved. At present, however, we are only able to establish the weaker inequality \eqref{eq:limit_pressure_ineq}.

Nevertheless, as mentioned in the introduction, \eqref{eq:complementarity} formally follows from \eqref{eq:continuity}, provided additional regularity assumptions hold. More precisely, assume that \(\alpha > 0\), and let \(t_0 \in [0,T]\) be such that \(p(t_0) \in BV(\mathbb{R}^d)\). Assume further that, for every compact set \(K \subset\subset \Omega(t_0)\), there exists \(\delta > 0\) such that
\[
\rho \in C^\infty\big((t_0-\delta,t_0+\delta)\times \mathring{K}\big).
\]

Let \(x \in \Omega(t_0)\), and choose a compact set \(K\) such that \(x \in K \subset\subset \Omega(t_0)\). By assumption, there exists \(\delta > 0\) such that \(\rho, p \in C^\infty\big((t_0-\delta,t_0+\delta)\times \mathring{K}\big)\). Since \(\rho(t_0,x) = 1\) and \(\rho \leq 1\), it follows that \(t_0\) is a local maximum of the function \(t \mapsto \rho(t,x)\), and therefore \(\partial_t \rho(t_0,x) = 0\). Moreover, \(\rho\) satisfies \eqref{eq:continuity} in the sense of distributions, and thus in the classical sense on \((t_0-\delta,t_0+\delta)\times \mathring{K}\). Hence,
\[
\underbrace{\partial_t\rho(t_0,x)}_{=0}
+ \big(\underbrace{\nabla \rho(t_0,x)}_{=0}\cdot U(x)\big)(1-p(t_0,x))
+ \underbrace{\rho(t_0,x)}_{=1}
\nabla \cdot\big((1-p(t_0))U\big)(x)
= 0,
\]
which yields the desired result.
\end{remark}

\begin{remark}
It is not reasonable to assume that $\rho$ and $p$ are globally smooth for the previous remark, as spatial discontinuities are expected on the back part $\partial \Omega(t_0)\setminus \partial \Omega_f(t_0)$. 
However, even though this is currently beyond reach from a rigorous standpoint, we expect that the problem retains a hyperbolic nature in the limit, so that information propagates at finite speed. In particular, it is natural to conjecture that when $x$ lies in the interior of the saturated region $\Omega(t_0)$, then it remains saturated for $t$ in a neighborhood of $t_0$.
\end{remark}

\section{Qualitative properties, entropies and shocks}
\label{sec:entropies}

Thanks to Theorem~\ref{thm:main}, we have shown that, up to the extraction of a subsequence, the solution \(\rho_{k,\varepsilon}\) to \eqref{eq:RegParaBis} converges to a solution of \eqref{eq:FBF}, thereby providing a rigorous mathematical formulation of the limiting problem. Notice that the estimates established in Lemma~\ref{thm:LemmaMultiD} are uniform with respect to both \(k\) and \(\varepsilon\), so that the order in which the limits are taken does not affect the analysis. In other words, we obtain a new macroscopic model with hard congestion constraints as a limit of an LWR-like model with diffusion.

A natural question is then to investigate the behavior of the corresponding \textit{hyperbolic} model \eqref{eq:CL} as the parameter \(k \to +\infty\). This model takes the form of a nonhomogeneous scalar conservation law. The following section is devoted to the study of \eqref{eq:CL}: we begin with several remarks on the one-dimensional case \(d=1\) with constant velocity field \(U\), which falls within the scope of the classical theory. We then derive the entropy inequalities in the multidimensional and non-homegeneous case, and analyze their behavior in the limit \(k \to +\infty\). Finally, we discuss the PDE structure of \eqref{eq:FBF}, as well as issues related to its well-posedness.

\subsection{Behavior in the unidimensional and homogeneous case}

Equation \eqref{eq:CL}, in the particular case $d=1$ and $U=1$—or more generally when $U$ does not depend on space—fits perfectly into the classical framework of scalar conservation laws. In this setting, the convergence of solutions of \eqref{eq:RegPara} toward a solution of \eqref{eq:CL} as $\eps \to 0$ is classical. However, for equations of this type, a weak formulation alone is not sufficient to ensure uniqueness. 

We recall well-known results from conservation laws, applied in this particular case. Consider the Cauchy problem
\begin{equation}
\label{eq:Riemann}
\begin{cases}
    \p_t\rho + \p_x(F_k(\rho))=0, \\
    \rho(t=0)=\rho^0,
\end{cases}
\end{equation}
where 
\[
\rho^0(x) = \begin{cases}
    \rho^- \quad \text{for } x < 0, \\
    \rho^+ \quad \text{for } x>0,
\end{cases}
\]
also known as the \textit{Riemann problem}, \cite{dafermos_constantine_m_hyperbolic_2016, serre_systemes_1996}. 

Then, \eqref{eq:Riemann} admits various weak solutions, among which the shock wave
\begin{equation}
\label{eq:shock}
\rho(t,x) = \left\{
\begin{array}{ccc}
    \rho^- & \text{for} & x <\sigma_k t, \\
    \rho^+ & \text{for} & x>\sigma_k t,
\end{array}
\right.
\end{equation}
where $\sigma_k$ satisfies the Rankine-Hugoniot relation $\sigma_k = [F_k(\rho)]/[\rho]$, using the standard notation $[a] = a^+-a^-.$ (see Fig.~\ref{fig:shock_rarefaction}). 

In the case of a \textit{jump-down} ($\rho^->\rho^+$), another weak solution is the rarefaction wave
\begin{equation}
\label{eq:rarefaction}
\bar{\rho}(t,x) = \left\{
\begin{array}{ccc}
    \rho^- & \text{if} & x<F_k'(\rho^-)t, \\
    (F_k')^{-1}\!\left(\f{x}{t}\right) 
        & \text{if} & F_k'(\rho^-)t<x<F_k'(\rho^+)t, \\
    \rho^+ & \text{if} & x> F_k'(\rho^+)t,
\end{array}
\right.
\end{equation}
which is well-defined since, by strict concavity of $F_k$, the function $F_k'$ is injective and $F_k(\rho^-)<F_k(\rho^+)$.

\begin{figure}[t]
\centering
\includegraphics[width = 0.8\linewidth]{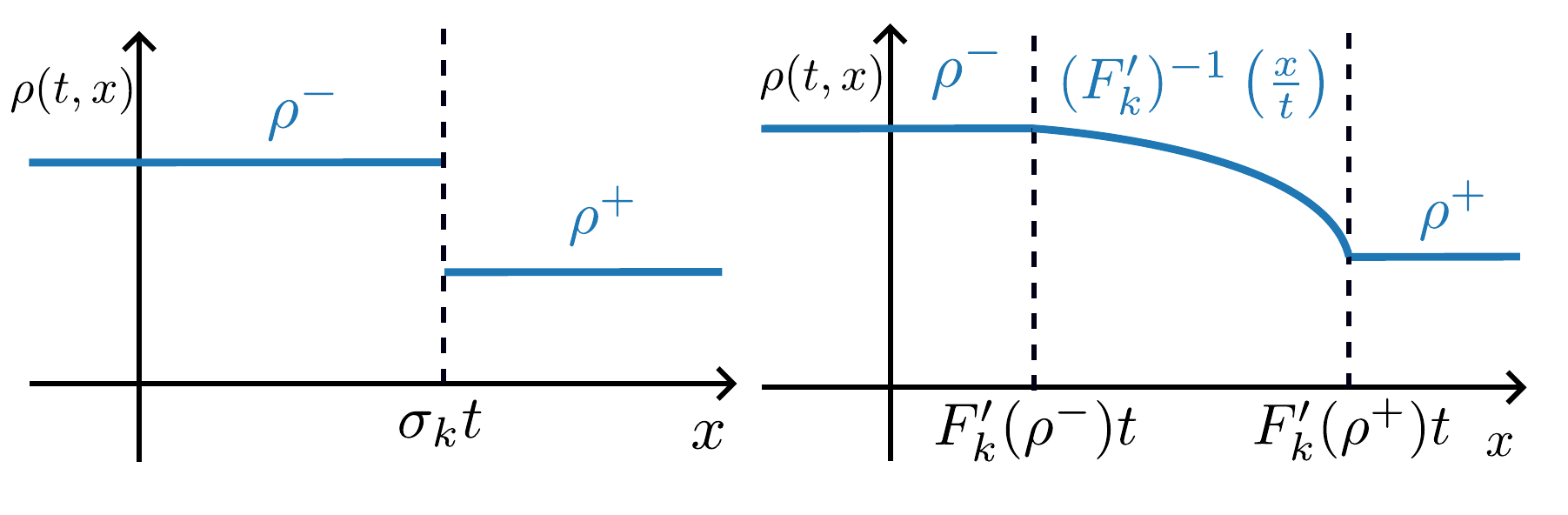}
\caption{Shock wave (left) and rarefaction wave (right)}
\label{fig:shock_rarefaction}
\end{figure}

To characterize the unique weak solution of \eqref{eq:CL} arising as the limit of solutions of \eqref{eq:RegPara}, one must supplement \eqref{eq:CL} with a family of differential inequalities, the so-called \textit{entropy inequalities} : for every $S$ convex and $Q_k$ such that $Q_k' = S'F_k'$

\begin{equation}
    \label{eq:entropy_1d}
    \p_t S(\rho) + \p_xQ_k(\rho) \le 0,
\end{equation}
and the solution is said to be entropic if it satisfies both \eqref{eq:Riemann} and \eqref{eq:entropy_1d}.

These inequalities act as a selection principle among weak solutions : in the case  $\rho^-<\rho^+$, then the shock wave \eqref{eq:shock} is entropic; and in the case $\rho^->\rho^+$, the rarefaction wave \eqref{eq:rarefaction} is entropic.

For the rest of the paper, we call  \textit{jump-ups} (resp. \textit{jump-downs}) shock waves \eqref{eq:shock} with $\rho^-<\rho^+$ (resp. $\rho^->\rho^+$). Since for finite $k$ \textit{jump-ups} are entropic while \textit{jump-downs} are not, one might expect this property to persist in the limit $k \to +\infty$, and hence for the free boundary formulation \eqref{eq:FBF}. More precisely, the need for entropy inequalities stems from the strict concavity of $F_k$. In the more general Lax–Oleinik framework, it is the nonlinearity of the flux that plays a decisive role. However, as $k\to+\infty$, the flux $F_k$ progressively flattens on $[0,1)$ (see Fig.~\ref{fig:flux}) in the sense that it approaches a linear function on each compact subset of $[0,1)$, and the nonlinearity disappears on this interval. This suggests that, in the limit $k\to+\infty$, equation \eqref{eq:FBF} behaves like a conservation law with a linear flux on $[0,1)$, meaning that \textit{jump-downs} between $1>\rho^->\rho^+$ could be admissible.

The case of a \textit{jump-down} with $\rho^-=1>\rho^+$ is more delicate. On the one hand, we could expect that such a discontinuity should not be admissible in the limit; on the other hand, as mentioned in the introduction, a shock propagating with speed $U(x^+)=1$ (see Fig.~\ref{fig:FBF1d}) appears to be the most natural solution of \eqref{eq:FBF}. Note also that for finite $k$, individuals located in the saturated region do not move, whereas in the limit, the entire saturated block in this example moves with speed $1$.

This phenomenon can be clarified by examining more closely the entropy solution of \eqref{eq:CL} in the case of a \textit{jump-down} with $\rho^-=1>\rho^+$. As discussed above, the unique entropy solution is the rarefaction wave \eqref{eq:rarefaction} with $\rho^-=1$. In this situation, $F_k'(\rho^-)= -k$, while for large $k$, $F_k'(\rho^+)\simeq 1$. 

Consider for instance the initial datum
\[
\rho^0 = \mathbf{1}_{[0, 1/2]} 
+ \f{1}{2}\mathbf{1}_{[1/2,1]}.
\]
Then (see Fig.~\ref{fig:rarefaction}), from the front of the saturated region ($x=1/2$), the rarefaction wave propagates in both directions: forward with velocity $1$, and backward with velocity $k$. For large $k$, the backward front reaches the rear of the saturated region almost instantaneously (at time $\tau=1/(2k)$), forcing the density to decrease and thereby allowing individuals to move again. 

At the second discontinuity located at $x=1$, a rarefaction wave also forms, but its two characteristic speeds, $1-(k+1)/2^k$ and $1$, are very close for large $k$.

The next paragraph is devoted to a more rigorous analysis of entropy solutions and admissible discontinuities for \eqref{eq:CL}, including the two-dimensional and non-homogeneous setting.

\begin{figure}[t]
\center
\includegraphics[scale = 0.5]{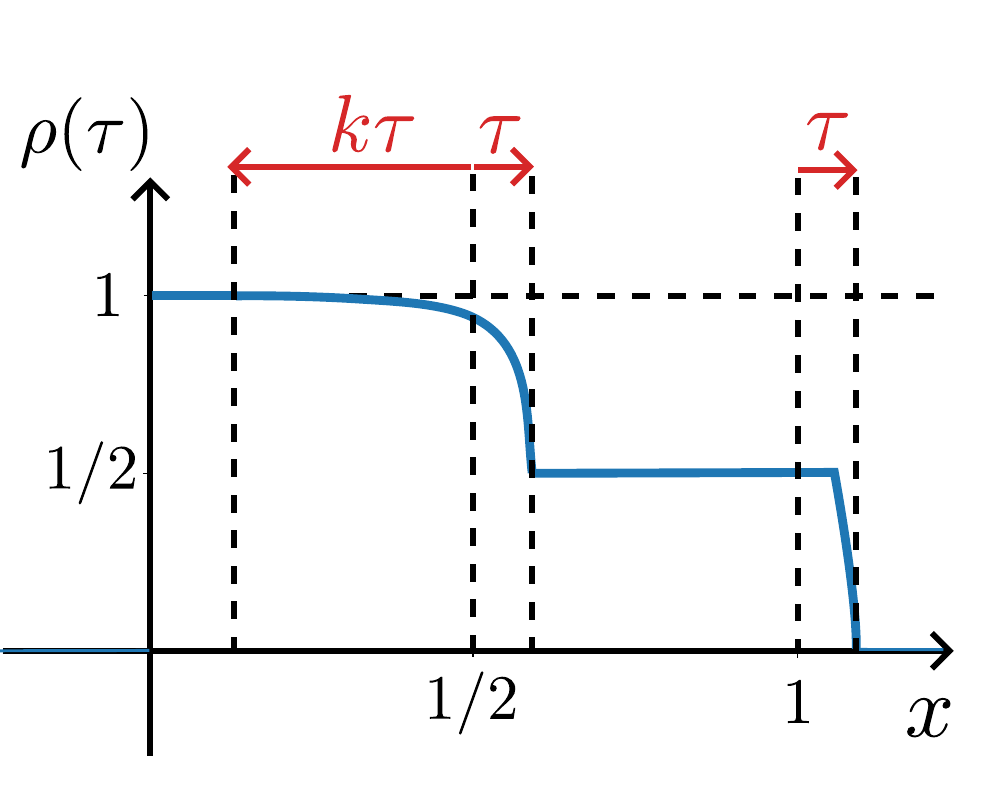}
\caption{Rarefaction wave}
\label{fig:rarefaction}
\end{figure}

\subsection{Entropies and discontinuities for finite $k$}

In the classical theory of conservation laws, entropy inequalities are a collection of differential inequalities that characterize which weak solution of \eqref{eq:CL} corresponds to the limit of solutions of the parabolic regularization \eqref{eq:RegPara}. The main theory being often presented for equations in dimension $1$ with a homogeneous flux, we establish those inequalities in a larger framework, adapted to our model.

\begin{proposition}[Entropies for finite $k$]
    Denote $\rho_{k,\eps}$ the solution of \eqref{eq:RegParaBis}. Then, for every $k > 0$, $\rho_{k,\eps} \underset{\eps \to 0}{\longrightarrow}\rho_k$ in $C([0,T],L^1_{loc}(\R^d))$. Moreover, for every entropy-flux pair $(S,Q_k)$, with $S$ a continuous convex function and $Q_k$ such that $Q_k' = F_k' S'$, $\rho_k$ satisfies in the distributional sense 

\begin{equation}
\label{eq:entropy_k}
    \p_t S(\rho_k) + \nabla \cdot (Q_k(\rho_k)U) + (S'(\rho)F_k(\rho_k)-Q_k(\rho_k))\nabla \cdot U \le 0.
\end{equation}
It is true in particular for the so called \textit{Kruzkhov entropies} $S(\rho) = |\rho-c|$, for which $Q_k(\rho) = (F_k(\rho)-F_k(c))\sgn(\rho-c)$, and the previous inequalities write 

\begin{equation}
    \label{eq:entropy_k_kruz}
    \p_t|\rho_k-c| + \nabla\cdot((F_k(\rho_k)-F_k(c))\sgn(\rho_k-c)U) \le - F_k(c)\sgn(\rho_k-c)\nabla \cdot U
\end{equation}
\end{proposition}

\begin{proof}
    Take $S$ a smooth convex function, and $Q_k$ an associated flux as above. Then, multiplying equation \eqref{eq:RegParaBis} by $S'(\rho_{k,\eps})$, and noting $\rho_{k,\eps}$ by $\rho$ for clarity, on gets 

    \[
    \p_t S(\rho) + Q_k'(\rho) \nabla \rho \cdot U + S'(\rho)F_k(\rho) \nabla \cdot U = \varepsilon S'(\rho) \Delta \rho.
    \]
    Computing $\Delta S(\rho)$ gives 

    \[
    \Delta S(\rho) = \nabla \cdot (S'(\rho)\nabla \rho) = S'(\rho)\Delta \rho + S''(\rho) \vert \nabla \rho \vert ^2,
    \]
    which yields
    \[
    \p_t S(\rho) + Q_k'(\rho) \nabla \rho \cdot U + S'(\rho)F_k(\rho) \nabla \cdot U = \varepsilon \Delta S(\rho) - \varepsilon S''(\rho)\vert \nabla \rho \vert^2,
    \]
    and thus, $S$ being convex 
    \[
    \p_t S(\rho) + Q_k'(\rho) \nabla \rho \cdot U + S'(\rho)F_k(\rho) \nabla \cdot U \le \varepsilon \Delta S(\rho),
    \]
    or equivalently 
    \[
    \p_t S(\rho) + \nabla \cdot(Q_k(\rho)U) + (S'(\rho)F_k(\rho)-Q_k(\rho))\nabla \cdot U \le \eps \Delta S(\rho).
    \]
    Then, taking to the limit $\eps \to 0$ and generalizing to $S$ continuous is classical (see \cite{serre_systemes_1996} \cite{malek_weak_1996}).
\end{proof}

Now, restricting ourselves to dimension $d = 1$, one can establish the standard Rankine-Hugoniot conditions that describe the evolution of discontinuties.

\begin{proposition}(Rankine-Hugoniot for finite $k$)
\label{prop:RHk}
    Let $\rho_k(t,x) = \rho_k^-(x,t)\mathbf{1}_{x\le x(t)}+\rho_k^+(x,t) \mathbf{1}_{x>x(t)}$, were $\rho_k^-,\rho_k^+, x$ are smooth functions. Suppose that $\rho_k$ satisfies equation \eqref{eq:CL} in the distributional sense. Note $[\rho](t) = \rho_k^+(t,x(t))-\rho_k^-(t,x(t))$ and $[F_k(\rho)](t) = F_k(\rho_k^+(t,x(t)))-F_k(\rho_k^+(t,x(t)))$. Then
    \begin{equation}
        \label{eq:RH}
        x'(t)= \frac{[F_k(\rho)](t)}{[\rho](t)}U(x(t)).
    \end{equation} 
    In the particular case of a discontinuity between a saturated region and an uncongested one, say $\rho_k^+ = 1$, one gets :
    \begin{equation}
        \label{eq:RHSat}
        x'(t) = -\frac{F_k(\rho_k^-(t,x(t))}{1-\rho_k^-(t,x(t))}U(x(t)).
    \end{equation}
\end{proposition}

\begin{proof}
    We may compute, in the distributional sense, 
    \[
    \p_t\rho_k = \p_t\rho_k^- \mathbf{1}_{x\le x(t)} +\p_t\rho_k^+\mathbf{1}_{x>x(t)} + x'(t)[\rho](t)\delta_{x(t)},
    \]
    and notice that
    \[
    \p_x(F_k(\rho)U) = \p_x(F_k(\rho_k^-)U)\mathbf{1}_{x \le x(t)} + \p_x(F_k(\rho_k^+)U)\mathbf{1}_{x >x(t)} + [F_k(\rho)](t)\delta_{x(t)}.
    \]
    Since $\rho_k$ satisfies \eqref{eq:CL}, this gives the desired equation on $x'(t)$.
\end{proof}

\begin{remark}
Equation \eqref{eq:RHSat} illustrates the behavior in the presence of a jam ahead. Since the constant state $1$ is a solution of \eqref{eq:CL}, individuals inside the jam do not move. Whenever the density behind the jam is nonzero, the rear of the jam propagates backward with velocity
\[
x'(t) = -\frac{F_k(\rho_k^-(t,x(t)))}{1-\rho_k^-(t,x(t))}\,U(x(t)).
\]
As $\rho_k^-(t,x(t))$ approaches $1$, this velocity tends to $-k\,U(x(t))$.
\end{remark}

\begin{proposition}[Entropic shocks for finite $k$]
\label{thm:entropic_shocks_k}
Let $\rho_k$ be as in Proposition~\ref{prop:RHk}, and assume that it satisfies both the conservation law \eqref{eq:CL} and the entropy inequalities \eqref{eq:entropy_k}. Then,it holds
\[
(\rho^+ - \rho^-)U(x(t)) \ge 0.
\]
\end{proposition}

\begin{proof}
Assume first that $\rho^- < \rho^+$ and take $\rho^- < c < \rho^+$.  
In this case, the entropy inequalities \eqref{eq:entropy_k_kruz} read
\[
\partial_t|\rho_k-c| 
+ \partial_x\!\big((F_k(\rho_k)-F_k(c))\sgn(\rho_k-c)\,U\big)
\le - F_k(c)\sgn(\rho_k-c)\,\partial_x U.
\]

Since $c$ lies between $\rho_k^-$ and $\rho_k^+$, we obtain locally around $x(t)$
\[
\partial_t|\rho_k-c|
= (-\partial_t\rho_k^-)\mathbf{1}_{x\le x(t)}
+ (\partial_t\rho_k^+)\mathbf{1}_{x>x(t)}
+ x'(2c-\rho_k^- -\rho_k^+)\delta_{x(t)},
\]
and
\[
\partial_x(Q_k(\rho_k)U)
= -\partial_x(Q_k(\rho_k^-)U)\mathbf{1}_{x\le x(t)}
+ \partial_x(Q_k(\rho_k^+)U)\mathbf{1}_{x>x(t)}
- [2F_k(c)-F_k(\rho_k^-)-F_k(\rho_k^+)]U(x(t))\delta_{x(t)},
\]
while
\[
- F_k(c)\sgn(\rho_k-c)\partial_x U
= F_k(c)(\partial_x U)\mathbf{1}_{x\le x(t)}
- F_k(c)(\partial_x U)\mathbf{1}_{x>x(t)}.
\]

The relevant contribution is the coefficient of the Dirac mass $\delta_{x(t)}$, which yields
\[
x'(2c-\rho_k^- -\rho_k^+)
- (2F_k(c)-F_k(\rho_k^-)-F_k(\rho_k^+))U(x(t))
\le 0.
\]

Since $\rho_k$ satisfies \eqref{eq:CL}, the velocity $x'$ fulfills the Rankine–Hugoniot condition \eqref{eq:RH}, namely
\[
x' = \frac{F_k(\rho_k^+)-F_k(\rho_k^-)}{\rho_k^+ - \rho_k^-}U(x(t)).
\]
Substituting this expression gives
\[
\frac{F_k(\rho_k^+)-F_k(\rho_k^-)}{\rho_k^+ - \rho_k^-}U(x(t))
\big[(c-\rho_k^-)-(\rho_k^+-c)\big]
- \big[(F_k(c)-F_k(\rho_k^-))-(F_k(\rho_k^+)-F_k(c))\big]U(x(t))
\le 0.
\]

Because $F_k$ is strictly concave and $\rho_k^- < c < \rho_k^+$,
\[
\frac{F_k(\rho_k^+)-F_k(c)}{\rho_k^+-c}
<
\frac{F_k(\rho_k^+)-F_k(\rho_k^-)}{\rho_k^+-\rho_k^-}
<
\frac{F_k(c)-F_k(\rho_k^-)}{c-\rho_k^-}.
\]
Using the notation $[a]=a^+-a^-$, this is equivalent to
\[
\begin{aligned}
F_k(\rho_k^+)-F_k(c)
&< \frac{[F_k(\rho_k)]}{[\rho]}(\rho_k^+-c),\\
F_k(c)-F_k(\rho_k^-)
&> \frac{[F_k(\rho_k)]}{[\rho]}(c-\rho_k^-).
\end{aligned}
\]

Hence
\[
U(x(t))\!\left[
\underbrace{\left(\frac{[F_k(\rho_k)]}{[\rho]}(c-\rho_k^-)
- (F_k(c)-F_k(\rho_k^-))\right)}_{<0}
+
\underbrace{\left((F_k(\rho_k^+)-F_k(c))
- \frac{[F_k(\rho_k)]}{[\rho]}(\rho_k^+-c)\right)}_{<0}
\right]
\le 0,
\]
which holds if and only if $U(x(t)) \ge 0$.

A similar computation in the case $\rho_k^- > \rho_k^+$ leads to $U(x(t)) \le 0$. This proves the stated property.
\end{proof}

\begin{remark}
This means that entropic shocks are precisely those that \textit{increase} in the direction of the velocity field $U$.
\end{remark}

\subsection{Entropies and discontinuities at the limit $k \to + \infty$}

\begin{proposition}[Entropies for the hard congestion hyperbolic problem]
Let $\rho_{k,\eps}$ be the unique solution of \eqref{eq:RegParaBis}.  
From the estimates \eqref{eq:BV} we have, up to a subsequence,
\[
\rho_{k,\eps} \to \rho \quad \text{in } C([0,T],L^1_{\mathrm{loc}}(\mathbb{R}^d)),
\]
and $\rho_{k,\eps}^k=p_{k,\eps} \overset{*}{\rightharpoonup} p$.  
Then, for all $S$ smooth and convex, the pair $(\rho,p)$ satisfies in the distributional sense

\begin{equation}
\label{eq:entropy_limit_gen}
    \p_t S(\rho)+\nabla \cdot [(S(\rho)-S'(\rho)p)U]\le  [S(\rho)-S'(\rho)\rho]\nabla \cdot U.
\end{equation}

This inequality remains true for the so-called \textbf{Kruzkhov entropies} : for all $c\in[0,1]$, the pair $(\rho,p)$ satisfies in the distributional sense
\begin{equation}
\label{eq:entropy_lim_kruz}
\partial_t |\rho-c|
+ \nabla\cdot\big((|\rho-c|-\sgn(\rho-c)p)U\big)
\le -c\,\sgn(\rho-c)\nabla\cdot U .
\end{equation}
\end{proposition}

\begin{proof}
Replacing $\rho_{k,\eps}$ by $\rho$ for simplicity, and using the previous computations, we know that for every smooth convex function $S$ and every $Q_k$ such that $Q_k' = F_k' S'$, the function $Q_k$ is defined up to a constant $c_k\in[0,1]$:
\[
Q_k(\rho)
= \int_{c_k}^{\rho} S'(s)F_k'(s)\,ds
= -\underbrace{\int_{c_k}^{\rho} S''(s)F_k(s)\,ds}_{\mathrm{I}}
+\underbrace{S'(\rho)F_k(\rho)}_{\mathrm{II}}
-\underbrace{S'(c_k)F_k(c_k)}_{\mathrm{III}}.
\]

Since $\rho_{k,\eps}\to\rho$ in $C([0,T],L^1_{\mathrm{loc}})$, we have almost everywhere convergence up to a subsequence.  
As $S$ is smooth, $S(\rho_{k,\eps})\to S(\rho)$ and $S'(\rho_{k,\eps})\to S'(\rho)$ almost everywhere.  
Because $p_{k,\eps}\overset{*}{\rightharpoonup}p$, we have
\[
F_k(\rho_{k,\eps})
=\rho_{k,\eps}(1-p_{k,\eps})
\rightharpoonup \rho(1-p).
\]

For every $x_k,c_k\to x,c\in[0,1]$,
\[
-\int_{c_k}^{x_k} S''(s)F_k(s)\,ds
\longrightarrow
-\int_c^x sS''(s)\,ds
= \int_c^x S'(s)\,ds - S'(x)x + cS'(c).
\]
By almost everywhere convergence of $\rho_{k,\eps}$ and dominated convergence, term (I) converges in $\mathcal{D}'$ to
\[
S(\rho)-S(c)-S'(\rho)\rho + S'(c)c.
\]
Term (II) converges in $\mathcal{D}'$ to $S'(\rho)\rho(1-p)$.

For term (III), we may choose $c_k\to 1$ so that $F_k(c_k)\to c\alpha$ for any $\alpha\in[0,1]$. However, this refinement does not affect the final result since the corresponding terms cancel. We then obtain
\[
Q_k(\rho_{k,\eps})
\to
S(\rho)-S(c)+cS'(c)-S'(\rho)p-c\alpha S'(c)
\quad\text{in }\mathcal{D}',
\]
with $\alpha=0$ for $c<1$ and $\alpha\in[0,1]$ for $c=1$.

Passing to the limit $k\to+\infty$, $\eps\to0$, we obtain for every smooth convex $S$:
\[
\partial_t S(\rho)
+\nabla\cdot\big((S(\rho)-S(c)+cS'(c)-S'(\rho)p-c\alpha S'(c))U\big)
\le
[S(\rho)-S(c)+cS'(c)-c\alpha S'(c)-S'(\rho)\rho]\nabla\cdot U.
\]
The terms involving $\alpha$ and $c$ cancel, and we finally get
\begin{equation}
\partial_t S(\rho)
+\nabla\cdot\big((S(\rho)-S'(\rho)p)U\big)
\le
[S(\rho)-S'(\rho)\rho]\nabla\cdot U.
\end{equation}

Taking now
\[
S(\rho)=H^\delta(\rho-c)
=\sqrt{(\rho-c)^2+\delta^2}-\delta,
\]
we obtain
\[
\partial_t H^\delta(\rho)
+\nabla\cdot\big((H^\delta(\rho)-(H^\delta)'(\rho)p)U\big)
\le
[(H^\delta)'(\rho)\rho-H^\delta(\rho)]\nabla\cdot U.
\]
Passing to the limit $\delta\to0$ gives in the distributional sense:
\[
\partial_t |\rho-c|
+\nabla\cdot\big((|\rho-c|-\sgn(\rho-c)p)U\big)
\le
[\sgn(\rho-c)\rho-|\rho-c|]\nabla\cdot U,
\]
which is exactly \eqref{eq:entropy_lim_kruz}.
\end{proof}

\begin{remark}
A key point is that we first use the entropy inequalities \eqref{eq:entropy_k} for smooth entropies $S$ and pass to the limit $k\to+\infty$ before extending the result to nonsmooth entropies. In this way, we avoid difficulties related to the lack of strong compactness of $p_k$.
\end{remark}

\begin{proposition}[Rankine–Hugoniot at the limit]
\label{thm:RH_infty}
Let
\[
\rho(t,x)
=\rho^-(x,t)\mathbf{1}_{x\le x^-(t)}
+\mathbf{1}_{[x^-(t),x^+(t)]}
+\rho^+(x,t)\mathbf{1}_{x\ge x^+(t)},
\]
where $\rho^\pm$ and $x^\pm$ are smooth functions and $\rho^\pm<1$.  
Assume that $\rho$ solves in the distributional sense
\begin{equation}
\label{eq:FBF1D}
\begin{cases}
\partial_t\rho+\partial_x(\rho(1-p)U)=0,\\
(1-p)U
=U\mathbf{1}_{x\le x^-(t)}
+U(x^+(t))\mathbf{1}_{[x^-(t),x^+(t)]}
+U\mathbf{1}_{x\ge x^+(t)},
\end{cases}
\end{equation}
that is, equation \eqref{eq:FBF} in the one-dimensional setting where the saturated region is $[x^-(t),x^+(t)]$. Then
\begin{equation}
\label{eq:RHlim}
\begin{cases}
(x^+)' = U(x^+),\\[0.2cm]
(x^-)' = \displaystyle\frac{U(x^+)-\rho^-U(x^-)}{1-\rho^-}.
\end{cases}
\end{equation}
\end{proposition}

\begin{proof}
Computing the derivatives in \eqref{eq:FBF1D}, we obtain
\[
\partial_t\rho
=\partial_t\rho^-\mathbf{1}_{x\le x^-}
+\partial_t\rho^+\mathbf{1}_{x\ge x^+}
-(x^-)'(1-\rho^-)\delta_{x^-}
-(x^+)'(1-\rho^+)\delta_{x^+},
\]
and
\[
\partial_x(\rho U)
=\partial_x(\rho^-U)\mathbf{1}_{x\le x^-}
+\partial_x(\rho^+U)\mathbf{1}_{x\ge x^+}
+[U(x^+)-\rho^-U(x^-)]\delta_{x^-}
+(1-\rho^+)U(x^+)\delta_{x^+}.
\]
Identifying the coefficients of the Dirac masses yields \eqref{eq:RHlim}.
\end{proof}

\begin{remark}
The second equation expresses the balance between two effects: the saturated zone moves with the desired velocity of its front $U(x^+)$, while the inequality $U(x^-)>U(x^+)$ implies that mass is added at the back of the saturated region.

If $\rho^-=0$, the endpoint in the back of the saturated area moves with velocity $U(x^+)$, as expected since no mass is added, and the entire saturated region travels at speed $U(x^+)$.  

Moreover, the balance law shows that above a certain density threshold behind the saturated zone, the rear of the queue propagates backward.  

Finally, since $U(x^-)>U(x^+)$, the relation implies that $\rho^-(t,x^-(t))$ must remain strictly less than~$1$, leading to a discontinuity at the back of the queue (see Fig.~\ref{fig:shock_formation}).
\end{remark}

\begin{figure}[t]
\centering
\includegraphics[width = 0.8\linewidth]{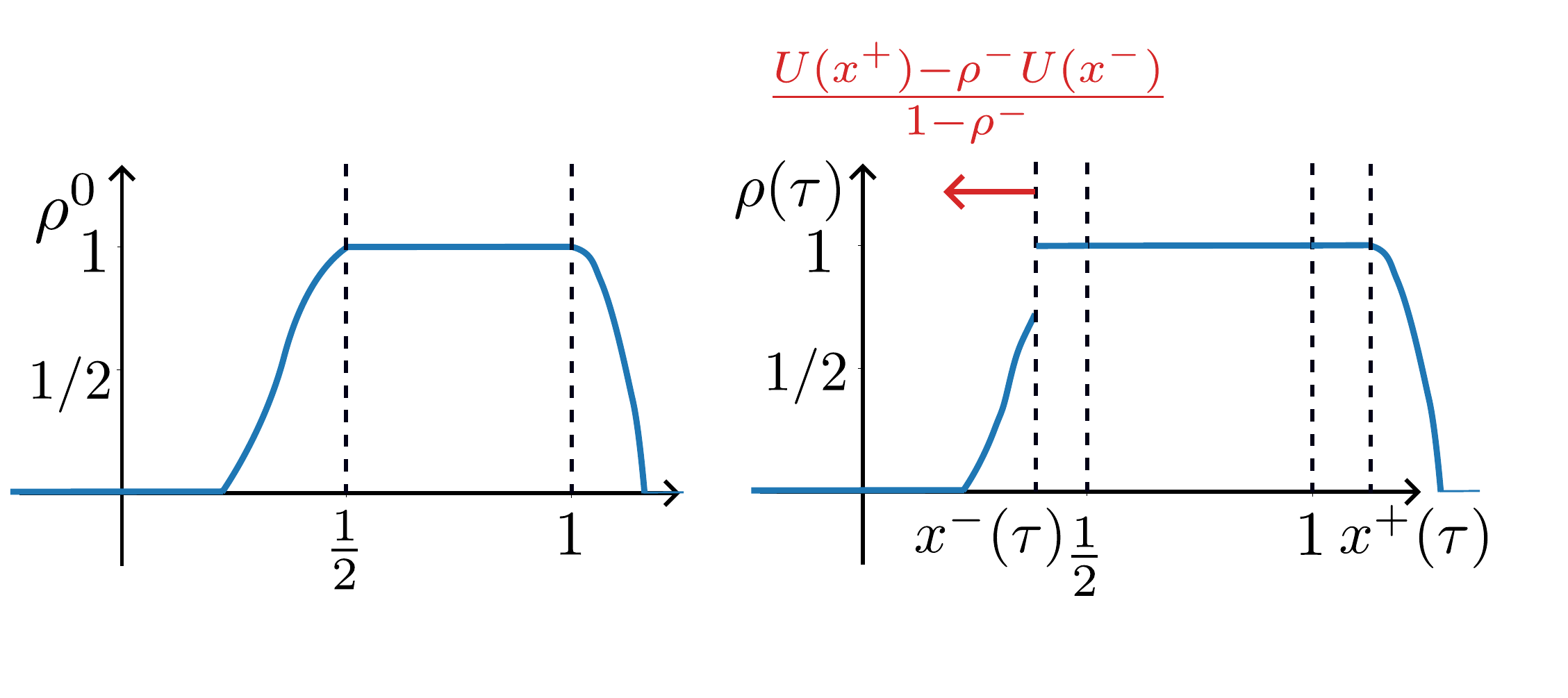}
\caption{Shock formation at the back of the saturated area.}
\label{fig:shock_formation}
\end{figure}

Let us now investigate which discontinuities satisfy both equation \eqref{eq:FBF} and the entropy inequalities \eqref{eq:entropy_lim_kruz}.  
It is not meaningful to consider a shock between two values $\rho^\pm<1$, since in that case $p=0$ and equation \eqref{eq:FBF} reduces to the standard linear transport equation, for which uniqueness follows from DiPerna–Lions theory \cite{diperna_ordinary_1989}.

\begin{proposition}[Entropic shocks at the limit]
\label{thm:entropic_shocks_infty}
Let $\rho$ be as in Proposition~\ref{thm:RH_infty}, satisfying \eqref{eq:FBF1D} in the distributional sense. Then $\rho$ also satisfies the entropy inequalities \eqref{eq:entropy_lim_kruz}.
\end{proposition}

\begin{proof}
Away from the discontinuities at $x^-(t)$ and $x^+(t)$, the entropy inequalities \eqref{eq:entropy_lim_kruz} are automatically satisfied. We therefore focus on these two interfaces.

\medskip
\noindent
\textbf{Discontinuity at $x^-$.}
Take $\rho^-<c<1$, which is the only nontrivial case.  
From \eqref{eq:FBF1D}, we have
\[
-pU=(U(x^+)-U(x^-))\mathbf{1}_{[x^-,x^+]}.
\]
Hence
\[
\partial_x\big((|\rho-c|-\sgn(\rho-c)p)U\big)
=
\partial_x(|\rho-c|U)(\mathbf{1}_{x\le x^-}+\mathbf{1}_{x\ge x^+})
+\big[U(x^-)(1+\rho^- -2c)-(U(x^-)-U(x^+))\big]\delta_{x^-}.
\]

Proceeding as in the proof of Proposition~\ref{thm:entropic_shocks_k} for $\partial_t|\rho-c|$ and collecting the coefficient of $\delta_{x^-}$, the entropy inequality becomes
\[
(x^-)'(2c-\rho^- -1)
+(1+\rho^- -2c)U(x^-)
-(U(x^-)-U(x^+))
\le 0.
\]
Using Proposition~\ref{thm:RH_infty},
\[
(x^-)'=\frac{U(x^+)-\rho^-U(x^-)}{1-\rho^-},
\]
so the inequality is equivalent to
\[
\frac{U(x^+)-\rho^-U(x^-)}{1-\rho^-}(2c-\rho^- -1)
+(1+\rho^- -2c)U(x^-)
-(U(x^-)-U(x^+))
\le 0.
\]
After rearrangement,
\[
(1+\rho^- -2c)\frac{U(x^-)-U(x^+)}{1-\rho^-}
\le
U(x^-)-U(x^+).
\]
Since $U$ is decreasing, $U(x^-)-U(x^+)\ge0$, and the entropy condition reduces to
\[
\frac{1+\rho^- -2c}{1-\rho^-}\le1,
\qquad \forall\,\rho^-<c<1,
\]
which clearly holds.

\medskip
\noindent
\textbf{Discontinuity at $x^+$.}
Take $\rho^+<c<1$.  
The key observation is that $p(t,x^+(t))=0$, so $\partial_x(\sgn(\rho-c)pU)$ produces no Dirac mass at $x^+$.  
Repeating the same computations and using $(x^+)'=U(x^+)$, the entropy inequality reduces to
\[
U(x^+)(2c-\rho^+ -1)
+(1+\rho^+ -2c)U(x^+)
\le 0,
\]
which is trivially satisfied.
\end{proof}

\begin{remark}
The previous proposition shows that both \emph{jump-ups} to $1$ and \emph{jump-downs} from $1$ are admissible under the entropy inequalities.  

For jump-ups this is expected, since \eqref{eq:entropy_lim_kruz} was constructed precisely to capture limits of entropy solutions of \eqref{eq:CL}, and the finite-$k$ entropy inequalities \eqref{eq:entropy_k} allow such discontinuities.  

The admissibility of jump-downs is more surprising and reflects the qualitative behavior illustrated in Fig.~\ref{fig:rarefaction}: for finite $k$, a jump-down generates a rarefaction wave whose backward propagation speed diverges as $k\to+\infty$.
\end{remark}

\begin{remark}
This also highlights that equation \eqref{eq:FBF} behaves quite differently from a classical conservation law. One might be tempted to interpret \eqref{eq:FBF} as a conservation law with a flux that is linear for $\rho\in[0,1[$ and highly nonlinear at $\rho=1$. However, in standard conservation laws, nonuniqueness arises precisely from flux nonlinearity. One could therefore expect that a weak formulation of \eqref{eq:FBF1D} would not suffice to ensure uniqueness, and that additional differential inequalities would be needed to single out the solution arising as the limit of entropy solutions to \eqref{eq:CL}.
\end{remark}

\section*{Acknowledgements}
The author FS is supported by the ERC AdG 101054420 EYAWKAJKOS. NM and FS also warmly acknowledge the support of the Lagrange Mathematics and Computation Research Center which hosted important discussions on this topic. The authors thank Bertrand Maury for the idea that originated the project, and also Frédéric Lagoutière from ICJ in Lyon for the numerics.
\newpage

\bibliographystyle{siam}
\bibliography{biblio}

\begin{figure}[t]
\centering

\begin{subfigure}{0.45\textwidth}
\includegraphics[width=0.9\linewidth]{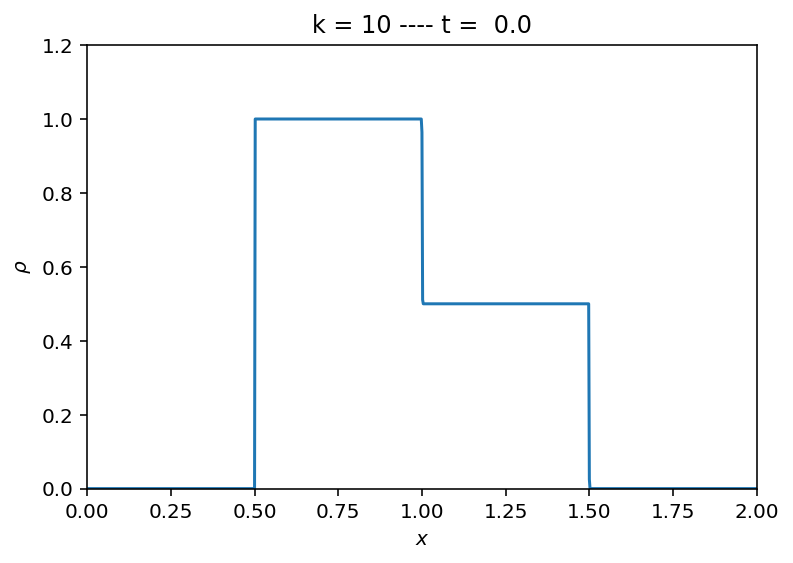}
\end{subfigure}

\medskip

\begin{subfigure}{0.45\textwidth}
\includegraphics[width=0.9\linewidth]{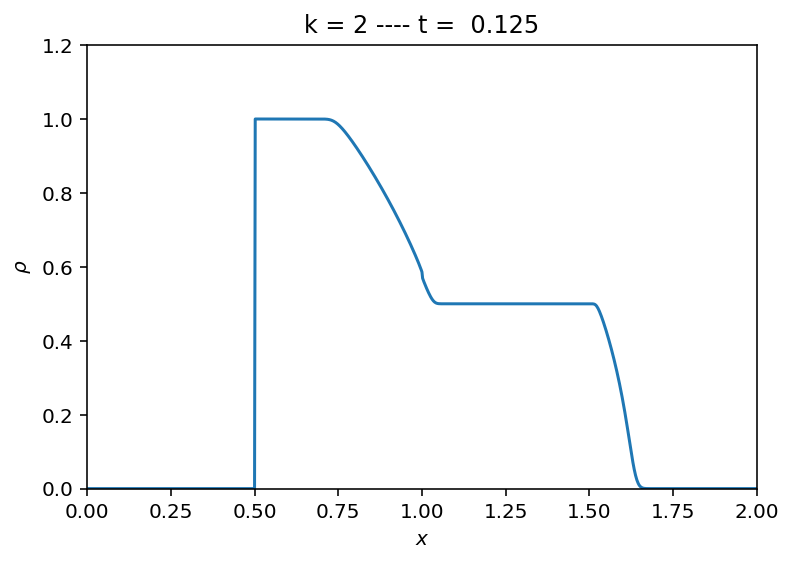}
\end{subfigure}
\hfill
\begin{subfigure}{0.45\textwidth}
\includegraphics[width=0.9\linewidth]{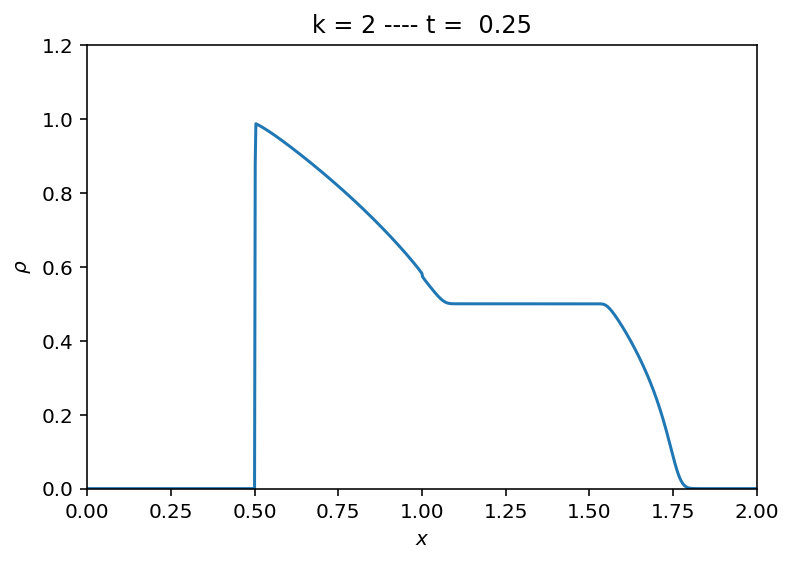}
\end{subfigure}

\medskip

\begin{subfigure}{0.45\textwidth}
\includegraphics[width=0.9\linewidth]{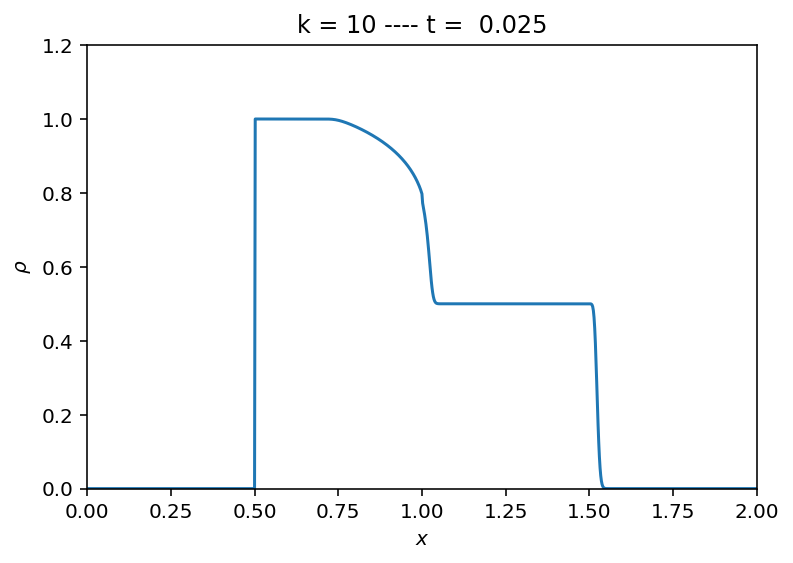}
\end{subfigure}
\hfill
\begin{subfigure}{0.45\textwidth}
\includegraphics[width=0.9\linewidth]{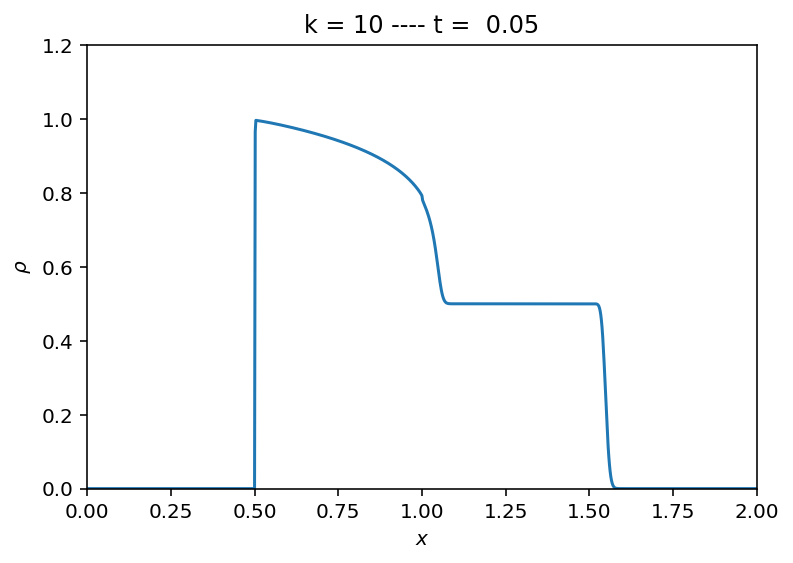}
\end{subfigure}

\medskip

\begin{subfigure}{0.45\textwidth}
\includegraphics[width=0.9\linewidth]{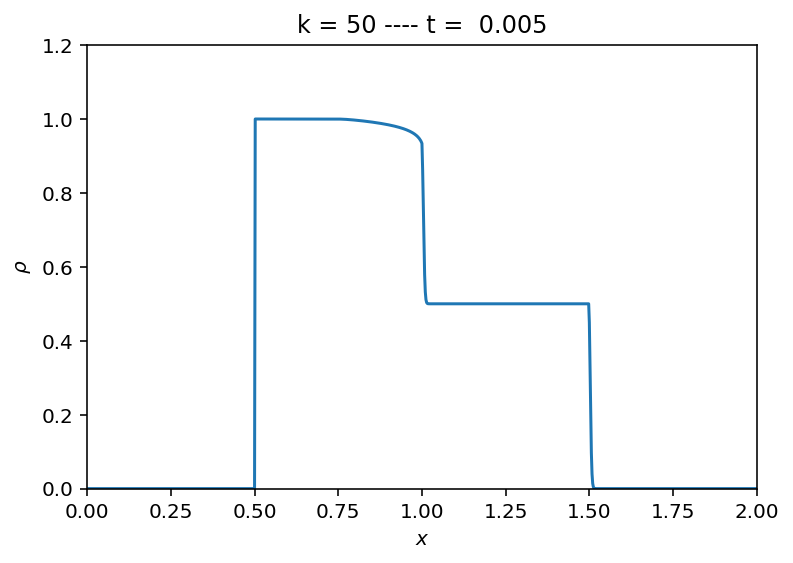}
\end{subfigure}
\hfill
\begin{subfigure}{0.45\textwidth}
\includegraphics[width=0.9\linewidth]{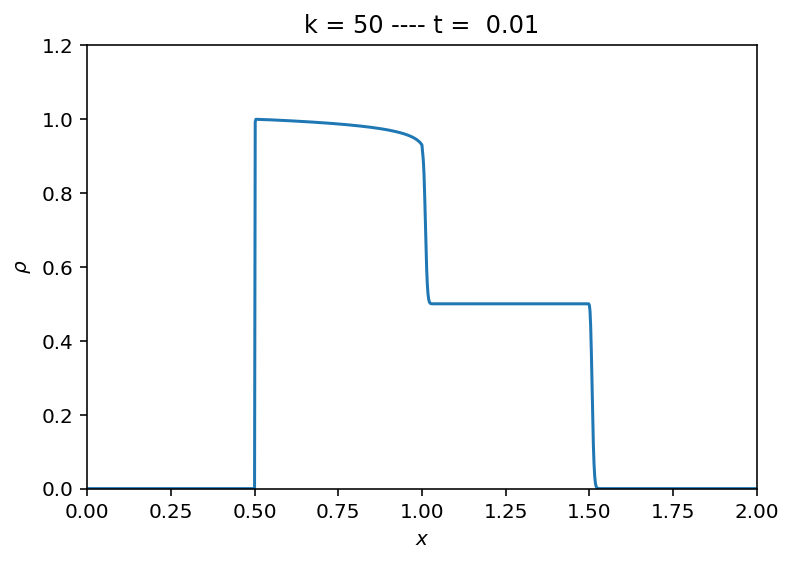}
\end{subfigure}

\caption{Numerical simulations with different values of $k$.}
\label{fig:CL1DkComparison}
\end{figure}

\newpage

\begin{figure}[t]
\centering
\begin{subfigure}{0.32\textwidth}
\includegraphics[width=\linewidth]{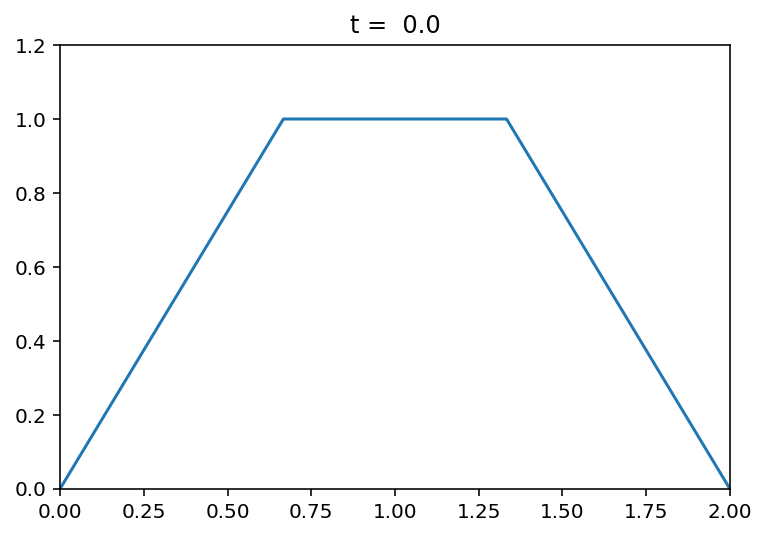}
\end{subfigure}
\hfill
\begin{subfigure}{0.32\textwidth}
\includegraphics[width=\linewidth]{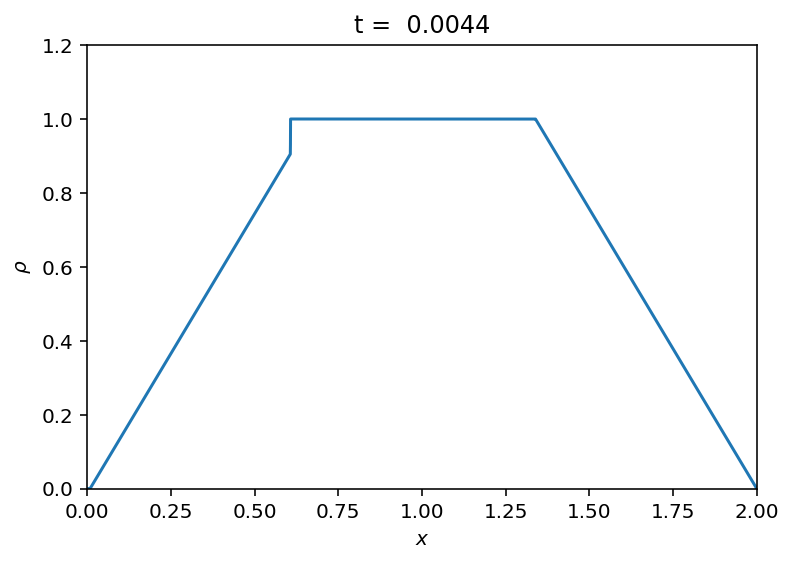}
\end{subfigure}
\hfill
\begin{subfigure}{0.32\textwidth}
\includegraphics[width=\linewidth]{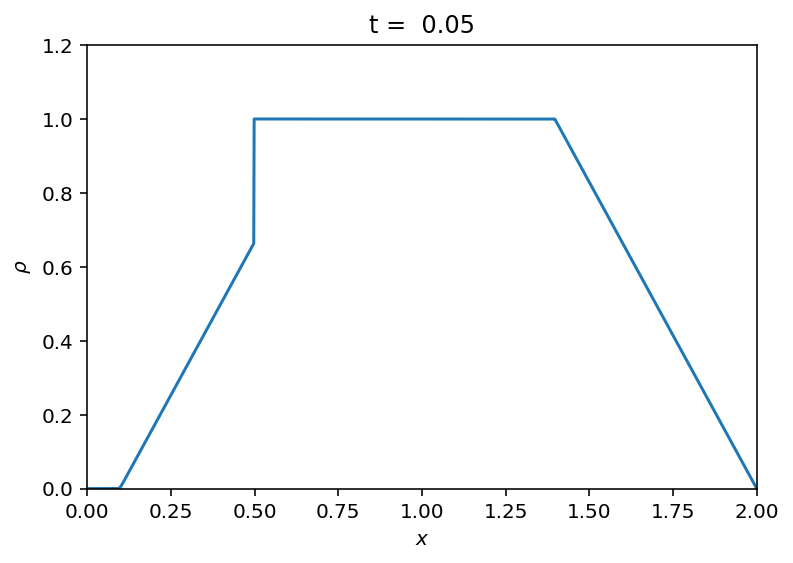}
\end{subfigure}
\caption{Shock formation for infinite $k$, numerical simulation.}
\end{figure}

\begin{figure}[t]
\centering

\begin{subfigure}{0.45\textwidth}
\includegraphics[width=\linewidth]{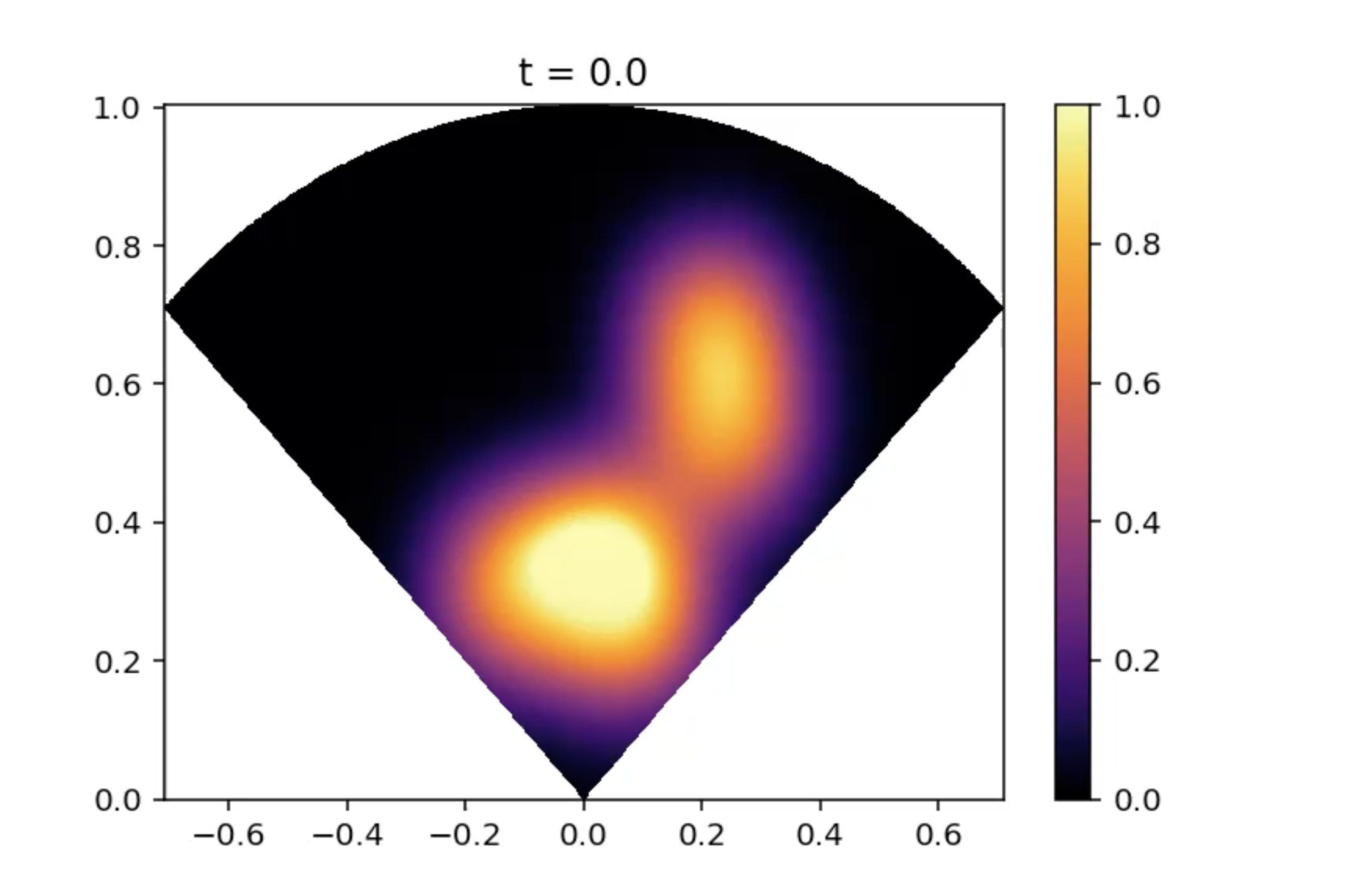}
\end{subfigure}
\hfill
\begin{subfigure}{0.45\textwidth}
\includegraphics[width=\linewidth]{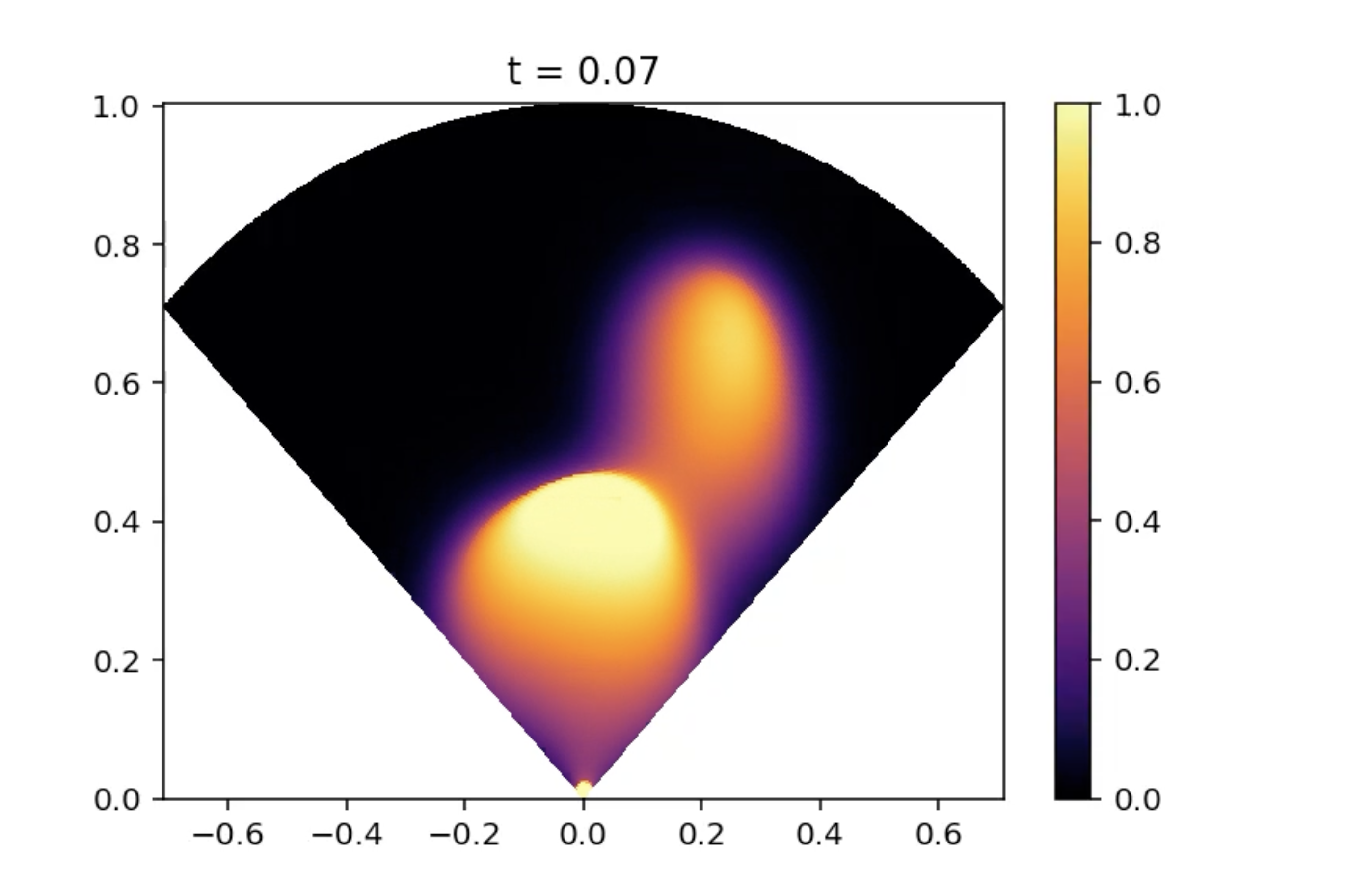}
\end{subfigure}

\medskip

\begin{subfigure}{0.45\textwidth}
\includegraphics[width=\linewidth]{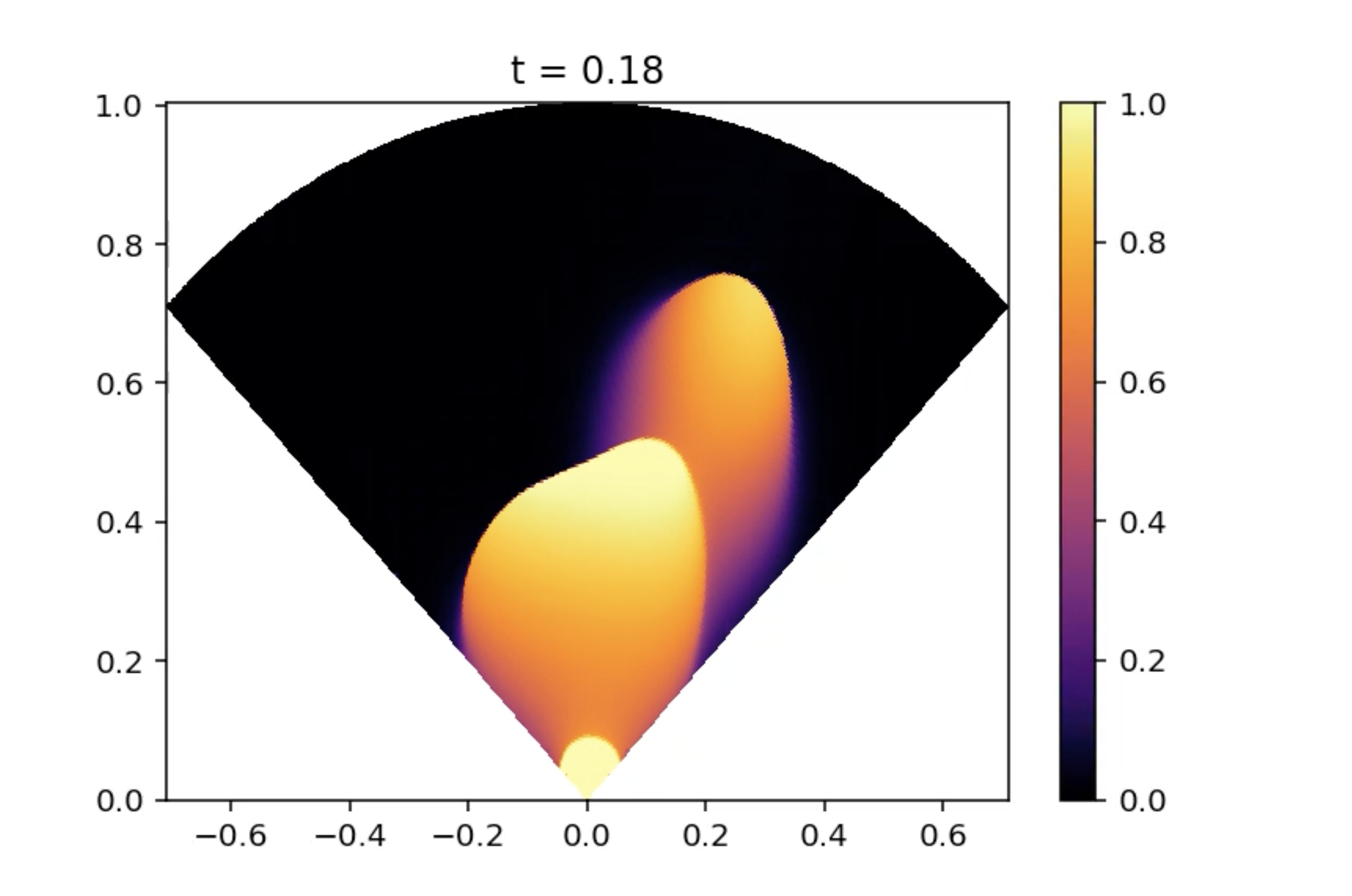}
\end{subfigure}
\hfill
\begin{subfigure}{0.45\textwidth}
\includegraphics[width=\linewidth]{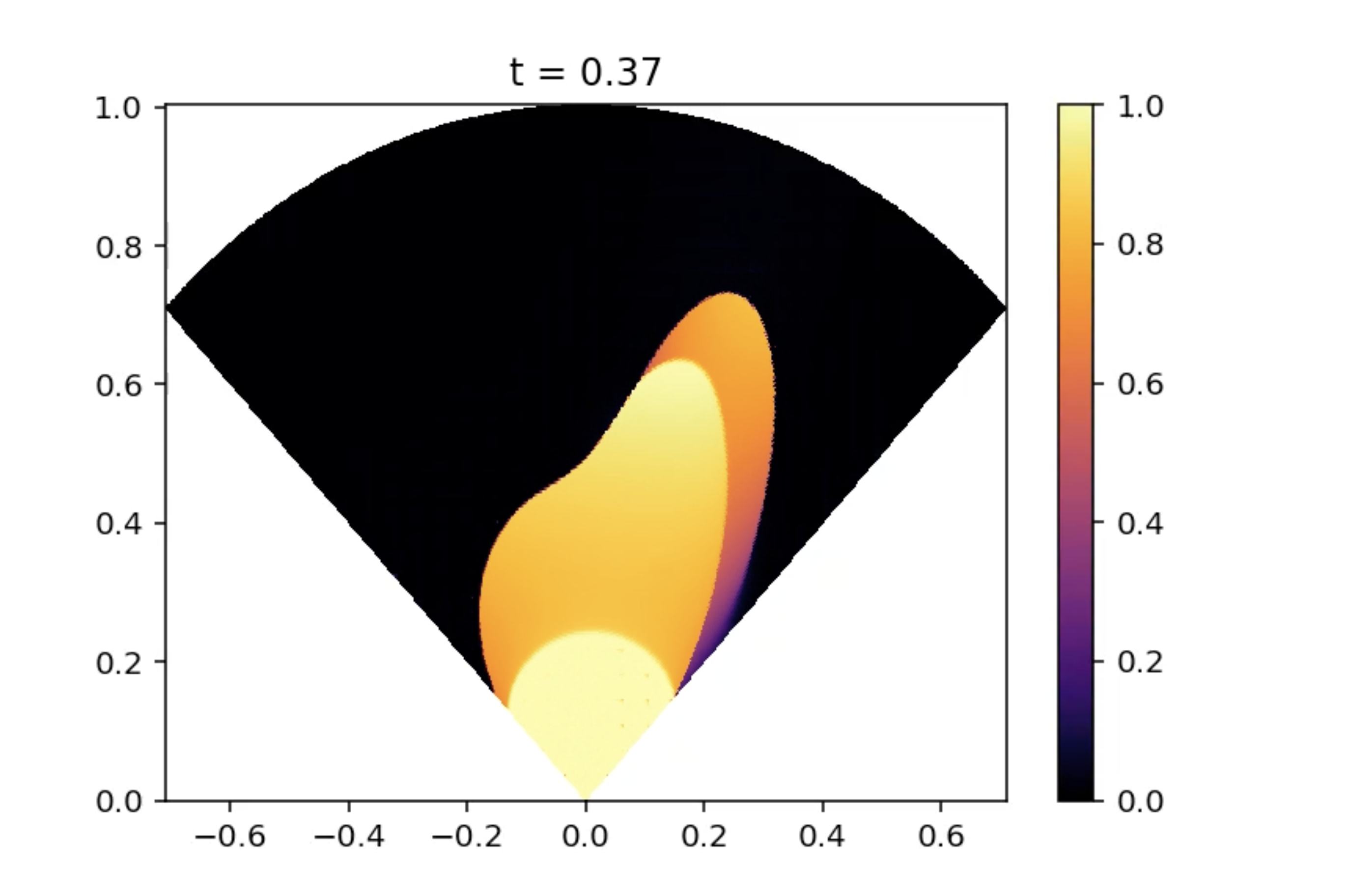}
\end{subfigure}

\medskip

\begin{subfigure}{0.45\textwidth}
\includegraphics[width=\linewidth]{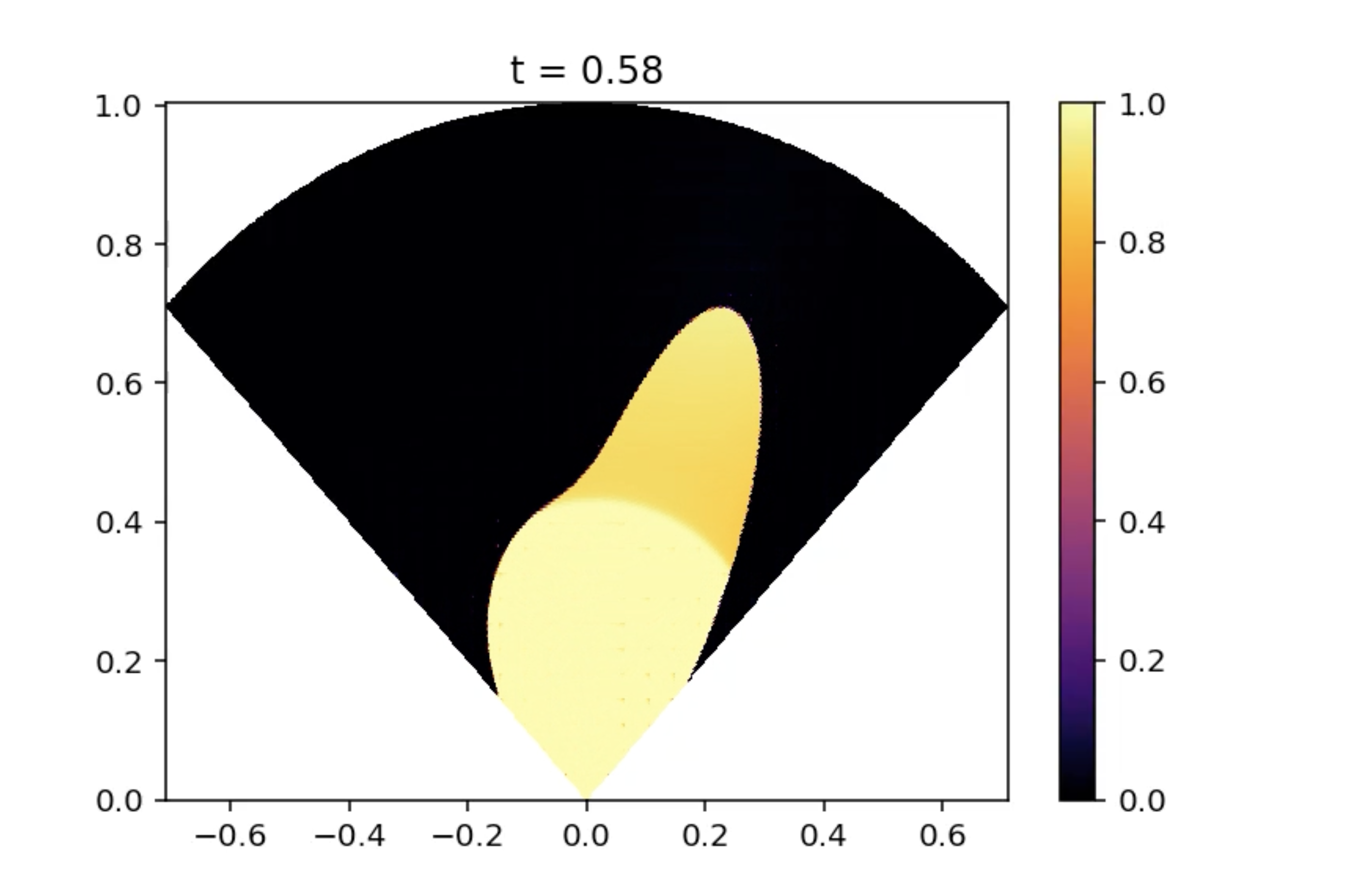}
\end{subfigure}
\hfill
\begin{subfigure}{0.45\textwidth}
\includegraphics[width=\linewidth]{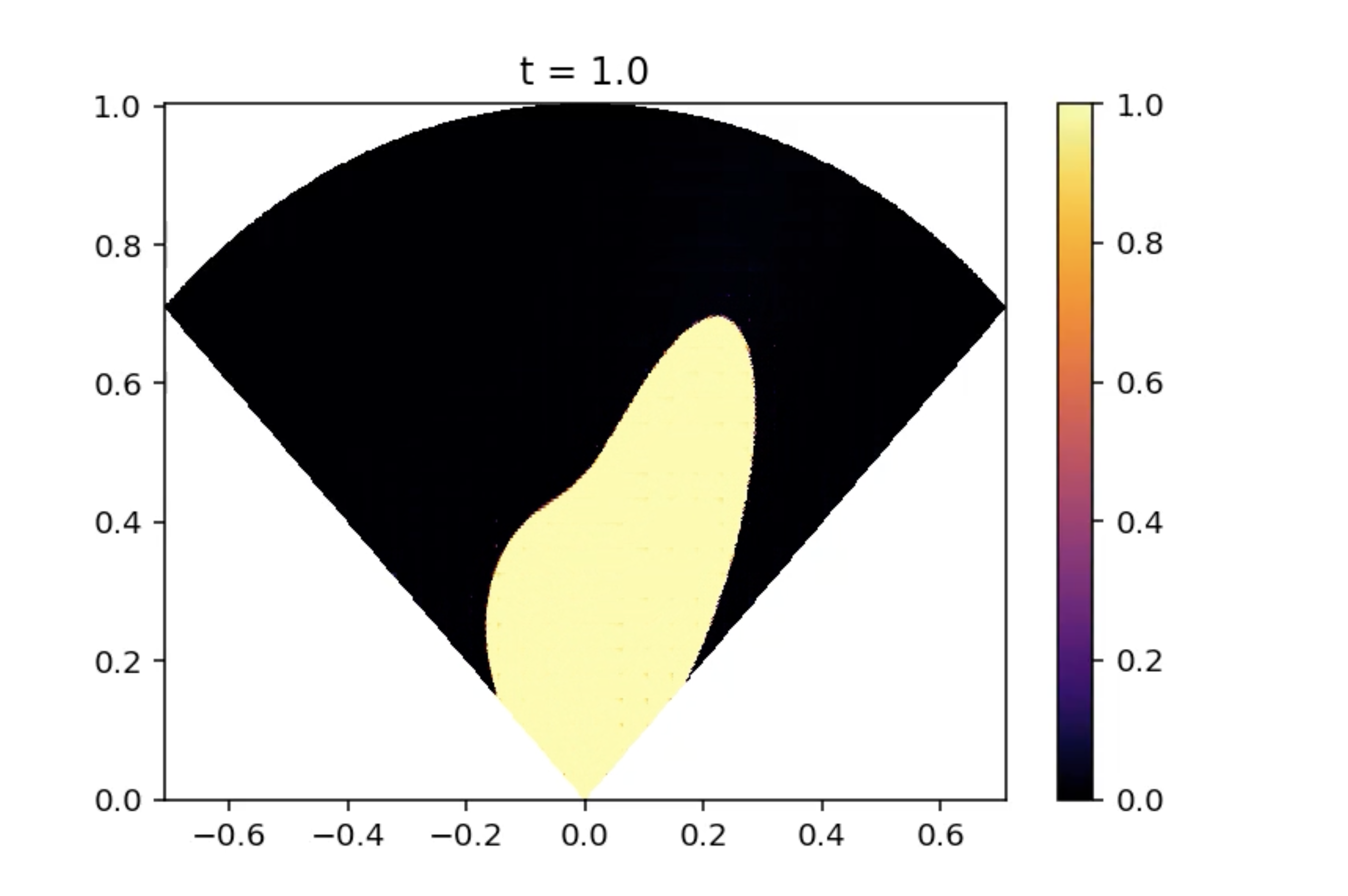}
\end{subfigure}
\caption{Numerical simulations in dimension $2$ for $k=1$.}
\label{fig:CL2Dk1_evol}
\end{figure}

\end{document}